\documentclass{amsart}
\usepackage{amsmath, amsfonts, amssymb,amsbsy,eucal,mathrsfs}
\usepackage[all]{xy}
\usepackage{pstricks}
\usepackage{pst-plot}
\usepackage{graphicx}
\usepackage[utf8]{inputenc}
\usepackage[T1]{fontenc}
\usepackage{chngcntr}

\newtheorem{thm}{Theorem}[subsection]
\counterwithin*{thm}{section}

\newtheorem{lemma}[thm]{Lemma}
\newtheorem{prop}[thm]{Proposition}
\newtheorem{cor}[thm]{Corollary}

\theoremstyle{definition}
\newtheorem{remark}[thm]{Remark}

\newtheorem{example}[thm]{Example}
\newtheorem{defn}[thm]{Definition}

\renewenvironment{proof}{{\flushleft \it Proof.}}{\hfill $\square$ \vspace{2mm}}

\DeclareMathOperator{\SL}{SL}
\DeclareMathOperator{\Sp}{Sp}
\DeclareMathOperator{\SO}{SO}
\DeclareMathOperator{\Ker}{Ker}

\newcommand{\Q}{{\mathbb Q}}

\newcommand{\cO}{{\mathcal O}}

\newcommand{\ignore}[1]{}

\def \a {{\alpha}}

\def \A {{\mathbb{A}}}

\def \im {{\rm Im}}

\newcommand{\G}{\mathbb{G}}

\renewcommand{\a}{{\alpha}}

\newcommand{\rk}{{\rm rk}}

\newcommand{\scal}[1]{\langle #1 \rangle}

\newcommand{\g}{\mathfrak g}

\renewcommand{\b}{\mathfrak b}

\newcommand{\ad}{{\rm ad}\ \!}

\newcommand{\haut}{{\rm ht}}

\def \cN {{\mathcal{N}}}

\newcommand{\kk}{{{\mathsf{k}}}}

\newenvironment{proof-of-prop-}{{\flushleft \it Proof of Proposition \ref{prop-irr=>rc}.}}{\hfill $\square$ \vspace{2mm}}

\begin{document}

\title{{Singularities of closures of spherical $B$-conjugacy classes of nilpotent orbits}}

\date{April 8, 2016}

\author{M.~Bender}
\address{Bergische Universit\"at Wuppertal, D-42097 Wuppertal, Germany}
\email{mbender@uni-wuppertal.de}

\author{N.~Perrin}
\address{Laboratoire de Math\'ematiques de Versailles, UVSQ, CNRS, Universit\'e Paris-Saclay, 78035 Versailles, France}
\email{nicolas.perrin@uvsq.fr}

\subjclass[2000]{Primary: 14M27,14M17,14M15. Secondary: 20G15, 14M12.}

\thanks{The first author thanks Magdalena Boos for helpful conversations about Borel conjugation of nilpotent elements. The second author was supported by a public grant as part of
  the Investissement d'avenir project, reference ANR-11-LABX-0056-LMH,
  LabEx LMH.}

\begin{abstract}
We prove that for a simply laced group, the closure of the Borel conjugacy class of any nilpotent
element of height $2$ in its conjugacy class is normal and admits a rational resolution. 
We extend this, using Frobenius splitting techniques, to the closure in the whole Lie algebra if either the group has type $A$ or the element has rank $2$.
\end{abstract}

\maketitle

\markboth{M.~BENDER \& N.~PERRIN}
{\uppercase{{Singularities of spherical nilpotent orbits}}}

\setcounter{tocdepth}{2}


\section*{Introduction}

We work over an algebraically closed field $\kk$. Let
$G$ be a reductive group and $\g$ be its Lie algebra. An element $x
\in \g$ is called nilpotent if ${\rm ad}_x$ is a nilpotent operator on
$\g$. This property is preserved under the adjoint action of the group
$G$ on $\g$. Nilpotent orbits are the orbits $G \cdot x$ under this
action when $x$ is nilpotent.  

These orbits have been classified via the existence of so called
$\mathfrak{sl}_2$-triples. More precisely, for $x$ nilpotent, there
always exists a $\mathfrak{sl}_2$-triple $(x,h,y)$ of elements in $\g$
such that the span of these elements is a subalgebra of $\g$
isomorphic to $\mathfrak{sl}_2$ and such that these elements satisfy
the Serre relations. In particular $h$ is semisimple. The eigenspace
decomposition of $h$ on $\g$ is of the form  $\g = \g(-n) \oplus \cdots \oplus \g(0) \oplus \cdots \oplus \g(n).$

\begin{defn}
The height $\haut(x)$ of $x$ is the maximal weight $n$ in the above
decomposition. It does not depend on the choice of a
$\mathfrak{sl}_2$-triple.    
\end{defn}

Since $x$ lies in $\g(2)$ we always have $\haut(x) \geq 2$. Nilpotent 
orbits which are spherical were first classified by
D. Panyushev \cite{Panyushev} in characteristic $0$. The results were
extented in positive characteristic by Fowler and R\"ohrle \cite{FR}. The 
classification is especially simple in terms of the height of $x$ if the characteristic is a good prime. 
Recall that a prime $p$ is good for $G$ if $p$ does not divide
any coefficient in the expression of roots in terms of simple roots
(see \cite[Definition 4.1]{SpringerSteinberg}). 

\begin{thm}[\cite{Panyushev},\cite{FR}]
Assume that the characteristic of $\kk$ is good for $G$. Then $G\cdot
x$ is spherical if and only if $\haut(x) \leq 3$. 
\end{thm}

The geometry of $B$-orbit closures is very simple in
$G$-varieties which are spherical. In this paper we consider $B$-orbit
closures in spherical nilpotent orbits. We focus on the case of height
$2$ since the $G$-orbits share a simple geometric structure in that
case: they are obtained via parabolic induction from Jordan
algebras (see Section \ref{induction}). Note also that in type $A$, all spherical elements have height $2$. 
Using this structure we prove the following result (see Theorem \ref{thm-origin}). 

\begin{thm}
Assume that the characteristic of $\kk$ is large enough. 
Let $G$ be a reductive group having only simply laced simple factors, let
$B$ be a Borel subgroup of $G$ and $x \in \g$ be nilpotent of height
$2$. Then the orbit closure $\overline{B \cdot x}$ in $\g$ is normal and has a
rational resolution outside the $G$-orbits it contains. Furthermore $\overline{B \cdot x}$ is homeomorphic to its normalisation.
\end{thm} 

\begin{cor}
For $x$ and $G$ as above, the closure of $B \cdot x$ in $G \cdot x$ is normal and has a rational resolution.
\end{cor} 

Our assumption on the characteristic of $\kk$ comes from the fact that we need the nilpotent orbit $\overline{G \cdot x}$ to be normal. Since this is true in characteristic zero, this also holds true for large enough characteristics. We give more precise bounds in the text. 

We actually produce a resolution of singularities with connected fibers of $\overline{B \cdot x}$ giving the last assertion of the above theorem and we are able to prove our regularity result outside $G$-orbits in $\overline{B \cdot x}$. Using Frobenius splitting techniques applied to the Jordan algebra of matrices or to the rank 2 Jordan algebras, we are able to extend the above result to the closure in the whole Lie algebra (see Corollary \ref{main-coro}):

\begin{thm}
Assume that the characteristic of $\kk$ is not $2$. Let $G$ be a reductive group having only simply laced simple factors, let
$B$ be a Borel subgroup of $G$ and $x \in \g$ nilpotent of height
$2$ such that $x$ is of rank $2$ in any simple factor not of type $A$. Then the closure in $\g$ of the $B$-orbit $B \cdot x$ is normal and has a rational resolution. 
\end{thm} 

The assumption on the characteristic comes from the fact that the nilpotent orbit $\overline{G \cdot x}$ is normal for $\kk$ of characteristic different from $2$ with the above assumption. Actually in type $A$ there is no assumption needed.

The proof goes as follows. By classical arguments for example
from \cite{FR}, it is easy to reduce to the case where $G$ is
simple. We then prove that for nilpotent orbits of height $2$, the
orbit can be obtained by parabolic induction from the
dense orbit of a Jordan algebra. For simply laced groups, this
dense orbit is of minimal rank. By a result of Ressayre \cite{ressayre}
it has a unique closed $B$-orbit and we prove that this $B$-orbit is
an affine space. Using classical techniques we construct this way
resolutions of singularities and prove our first result. Constructing Frobenius splittings in type $A$ and in rank $2$ we get the second result above. We also give
explicit descriptions of $B$-orbit closures for classical groups and
give an example where the $B$-orbit closures are non normal and 
non Cohen-Macaulay when $G$ is of type $C$ and obtain more 
normality results in type $B$.  


\section{$B$-orbits in spherical varieties}

In this section, we recall classical results on the structure of
$B$-orbits in spherical varieties (see for example \cite[Section
  4.4]{survey}). First recall the following result due to Brion
\cite{brion} and Vinberg \cite{vinberg}. 

\begin{prop}
A spherical variety has finitely many $B$-orbits.
\end{prop}

Let $Z$ be a spherical $G$-variety. We denote by $B(Z)$ its finite
set of $B$-orbits. 

\subsection{Weak order}

We define the weak Bruhat order on $B(Z)$ as follows. Let $Y \in
B(Z)$. If $P$ is a minimal parabolic subgroup of $G$ containing $B$
such that $Y \subsetneq PY$, we write $Y'$ for the dense $B$-orbit in
$PY$ and say that $P$ raises $Y$ to $Y'$. In this case we write $Y \preccurlyeq
Y'$. These relations are the covering relations of the weak Bruhat
order. For $Y \preccurlyeq Y'$ raised by $P$, consider the proper 
morphism $q : P \times^B Y \to PY$. The following result is proved in
\cite{RS1} for symmetric varieties and extends readily to the general
case. 

\begin{prop}\label{propTypesOfEdges}
\label{prop-UNT}
One of the following occurs.
\begin{itemize}
\item Type $U$: $PY = Y \cup Y'$ and $q$ is birational.
\item Type $N$: $PY = Y \cup Y'$ and $q$ has degree 2.
\item Type $T$: $PY = Y \cup Y'\cup Y''$ with $\dim Y''=\dim Y$ and $q$ is birational.
\end{itemize}
\end{prop}

\subsection{Varieties of minimal rank} Recall the following
definition. 

\begin{defn}\label{dfnMinimalRank}
A $G$-spherical variety is called of minimal rank if the covering
relations are only of type $U$. 
\end{defn}

\begin{prop}{\cite{ressayre}}
Any $G$-orbit of minimal rank contains a unique minimal $B$-orbit for
the weak order. 
\end{prop}

\begin{lemma}\label{lemma-resol}
Let $Z$ be a $G$-spherical variety and let $x \in Z$ such that $G
\cdot x$ has minimal rank. Let $Y= \overline{B.x}$ and let $Y_0$ be
the closure of the minimal $B$-orbit in $G \cdot x$.  
Then there exists a sequence of minimal parabolics $P_1 , \ldots ,
P_m$ such that  
$$f: P_1 \times^B \ldots \times^B P_m \times^B Y_0 \to Y , \quad [p_1
  , \ldots ,p_m , x] \mapsto p_1 \ldots p_m  \cdot x  ,$$ 
is birational and projective. In particular, $f$ is a resolution of
singularities if and only if $Y_0$ is a smooth variety. 

If furthermore all $G$-orbits of $Z$ are of minimal rank, then the above resolution has connected fibers.
\end{lemma}

\begin{proof}
Since the dense $B$-orbits of $Y$ and $Y_0$ are in $G \cdot x$ and since
the dense $B$-orbit in $Y_0$ is the only minimal orbit for the weak
order in $G \cdot x$, there exists a sequence of minimal parabolic
subgroups 
$P_1,\cdots,P_m$ raising $Y_0$ to $Y$. Since case (N) does not occur,
we see that the surjective and projective map
$$f: P_1 \times^B \ldots \times^B P_m \times^B Y_0 \to Y$$
is birational. The map $f$ is a resolution of singularities if and
only if $P_1 \times^B \ldots \times^B P_m \times^B Y_0$ is smooth.  

We now prove that $f$ has connected fibers. By induction it suffices to prove that for any $B$-orbit $Y'$ in $Y$ which is raised by a parabolic $P$, the map $q : P \times^B Y' \to P \cdot Y'$ has connected fibers.  We have the equality 
$q^{-1}(x) \simeq \left\{ p \in P \mid p^{-1} \cdot x \in Y' \right\}/B \subseteq P / B$. Note that $q$ is $P$-equivariant. If $P \cdot x \subseteq Y'$, then $q^{-1}(x) = P/B$ is connected. Else, since we are in type $U$, $q$ is birational and $x$ is in the dense open orbit of $PY'$ therefore $q^{-1}(x)$ consists of one point.
\end{proof}


\section{Nilpotent orbits of height $2$}
\label{induction}
Most of the results on nilpotent orbits of height $2$ in this section
are based on the paper \cite[Section 3]{FR} (see also
\cite{Panyushev}). 

\subsection{Reductive case}

Let $G$ be a reductive group. Since the kernel of the adjoint action
is the center of $G$, we can mod out the center and assume that $G$ is
semisimple of adjoint type. Applying Lemma 3.2 from \cite{FR} we get
the following result. 

\begin{lemma}
The group $G$ is a product $G = G_1 \times \cdots \times G_r$ of
simple groups and any nilpotent element $x \in \g$ is a sum $x = x_1 +
\cdots + x_r$ of nilpotent elements $x_i \in \g_i$. 
\end{lemma}

\begin{cor}
\label{cor-red}
The $B$-orbit closures in $\overline{G \cdot x}$ are products of orbit
closures of Borel subgroups in nilpotent orbits of simple groups. 
\end{cor}

In particular it is enough to deal with the case of simple groups. From now on we therefore assume that the group $G$ is simple.

\begin{lemma}
Let $x \in \g$ with $\haut(x) = 2$. Then $\haut(y) = 2$ for any $y \in
\overline{G \cdot x} - \{0\}$.
\end{lemma}

\begin{proof}
Since $G \cdot y$ lies in the closure of $G \cdot x$, the order of
nilpotency of $\ad_y$ is smaller than that of $\ad_x$. But $\ad_x$ has
order at most $3$ and elements of height bigger than $2$ have order at
least $4$. 
\end{proof}

\subsection{Structure of the $G$-orbits}

 Let $x \in \g$ be nilpotent of height $2$. Then
using a $\mathfrak{sl}_2$-triple, we have a decomposition of the Lie
algebra $\g = \g(-2) \oplus \g(-1) \oplus \g(0) \oplus \g(1) \oplus
\g(2)$ with $x \in \g(2)$. Let $P$ be the parabolic subgroup of $G$
with Lie algebra $\g(0) \oplus \g(1) \oplus \g(2)$. Its unipotent
radical $R_u(P)$ has Lie algebra $\g(1) \oplus \g(2)$. Let $L$ be the
Levi factor of $P$ with Lie algebra $\g(0)$.

The closure of $G\cdot x$ in $\g$ is $G \cdot \g(2)$ (see \cite[Lemma
  2.31]{FR}). We also have $C_G(x) = C_P(x)$ (see \cite[Theorem
  2.14]{FR}). We get an isomorphism $G \cdot x = G/C_G(x) = G/C_P(x)
= G \times^P P/C_P(x)$. Furthermore, $R_u(P)$ acts trivially on $x$
and $C_P(x)$ is the semidirect product of $C_L(x)$ with $R_u(P)$ (see
\cite[Proposition 2.24]{FR}). We get an isomorphism $G \cdot x = G
\times^P L/C_L(x)$. Since $L \cdot x$ is dense in $\g(2)$ (see
\cite[Corollary 2.21]{FR}) we get a birational morphism
$$G \times^P \g(2) \to \overline{G \cdot x}$$
given by $[g,y] \mapsto g \cdot y$. This morphism is an isomorphism on
the $G$-orbit which is therefore obtained from $L \cdot x$ via
parabolic induction: $G \cdot x \simeq G \times^P L \cdot x$.

\subsection{Structure of the $L$-orbits}
\label{L-orbit}

In this subsection we want to understand the possible pairs
$(L,\g(2))$ occuring in the above discussion. Consider the Lie
subalgebra $\g_E = \g(-2) \oplus \g(0) \oplus  \g(2)$. There is a
closed reductive subgroup $G_E$ of $G$ with Lie algebra $\g_E$. It
contains a parabolic subgroup $P_E$ whose Lie algebra is $\g(0) \oplus
\g(2)$. Furthermore, the unipotent radical $R_u(P_E)$ of $P_E$ has Lie
algebra $\g(2)$ and $P_E$ contains $L$ as a Levi subgroup. Note that
the unipotent radical of $P_E$ is abelian. This is quite restrictive,
indeed the type of the pair $(G_E,P_E)$ is one the pairs given in the
following list (we eventually mod out factors acting trivially and
give a list modulo isomorphism of the Dynkin diagram). In this list
$P_m$ corresponds to the maximal parabolic subgroup associated to the
$m$-th node of the Dynkin diagram with notation as in Bourbaki
\cite{bourbaki}: $(A_n,P_m)$, $(B_n,P_1)$,$(C_n,P_n)$, $(D_n,P_1)$,
$(D_n,P_n)$, $(E_6,P_6)$, $(E_7,P_7)$. 

An easy check proves that the cases $(A_n,P_m)$ with $m \neq n - m$,
the cases $(D_n,P_n)$ with $n$ odd and the case $(E_6,P_6)$ never
occur. Indeed, if we compute the decomposition of the
$\mathfrak{sl}_2$-triple $(y,h,x)$ in $\g_E$, we get $\g_E(1) \neq 0$
for these cases. A contradiction.  
In the next table we list the possibilities. In the last column we
describe the centraliser of our nilpotent element $x \in \g(2)$
(recall that $x$ has a dense $L$-orbit in $\g(2)$ -- these stabilisers
are well known, see for example \cite{kac} or \cite{ressayre}). 

\vskip 0.1 cm

\renewcommand{\arraystretch}{1.2}
$$\begin{array}{c|c|c|c|c|c}
\textrm{Case} & G_E & P_E & L & \g(2) & C_L(x) \\
\hline
1 & A_{2n-1} & P_n & A_{n-1} \times A_{n-1} & M_n(\kk) & A_{n-1} \\
2 & B_{n} & P_1 & B_{n-1}  & \kk^{2n-1} & B_{n-2} \\
3 & C_{n} & P_n & A_{n-1} & M^s_n(\kk) & B_{\frac{n-1}{2}} \textrm{ or
} D_{\frac{n}{2}} \\ 
4 & D_{n} & P_1 & D_{n-1} & \kk^{2n-2} & D_{n-2} \\
5 & D_{2n} & P_{2n} & A_{2n-1} &  M^a_{2n}(\kk) & C_{n} \\
6 & E_7 & P_7 & E_6 & \kk^{27} & F_4 \\
\end{array}$$	
\vskip 0.1 cm
\centerline{Table 1. Pairs $(G_E,P_E)$.}

\vskip 0.2 cm

\noindent
In the above table, we write $M_n(\kk)$ for the Jordan algebra of
square matrices of size $n$ over $\kk$, $M^s_n(\kk)$ for the
subalgebra of symmetric matrices and $M^a_n(\kk)$ for the subalgebra of
antisymmetric matrices. The vector space $\kk^{27}$ is endowed with
the structure of the unique exceptional semisimple Jordan algebra (of
antisymmetric matrices of rank $3$ over the octonions). Adding a unit
to the vector space $\kk^{n-2}$ we get the rank $2$ Jordan algebra
associated to any non-degenerate quadratic form on $\kk^{n-2}$.  

\subsection{Minimal rank and resolutions}

\begin{lemma}
If $G_E$ is simply laced, the $L$-orbit $L \cdot x$ in $\g(2)$ is of minimal rank. 
\end{lemma}

\begin{proof}
Follows from the above description and \cite{ressayre}.  
\end{proof}

\begin{cor}
If $G$ is simply laced, for any $x \in \g$
nilpotent of height $2$, the $G$-orbit in $\overline{G \cdot x}$ is
of minimal rank.  
\end{cor}

\begin{proof}
Let $y \in \overline{G \cdot x}$. Then $\haut(y) = 2$. Replacing
$x$ by $y$, we may assume $y = x$. Since $G \cdot x = G \times^P L
\cdot x$, it is enough to prove that $L \cdot x$ is of minimal
rank (see \cite[Lemma 6]{brion1}). By the above lemma, it is enough to
prove that $G_E$ is simply laced. But the Dynkin diagram of $G_E$ is a subdiagram
of that of $G$. 
\end{proof}

\begin{cor}
\label{cor-resol2}
Let $G$ be a simple simply laced group, let $x \in \g$
nilpotent of height $2$ and let $Y$ be a $B$-orbit closure.

Then there exists a unique minimal $B$-orbit $Y_0$ in the dense
$G$-orbit of $GY$ and a sequence of minimal parabolics $P_1 , \ldots ,
P_m$ such that  
$$f: P_1 \times^B \ldots \times^B P_m \times^B Y_0 \to Y , \quad [p_1
  , \ldots ,p_m , x] \mapsto p_1 \ldots p_m  \cdot x  ,$$ 
is birational, projective with connected fibers. In particular, $f$ is a resolution of 
singularities if and only if $Y_0$ is a smooth variety. 
\end{cor}

\begin{proof}
By Lemma \ref{lemma-resol}, it is enough to check that for $y$ in the
dense $B$-orbit of $Y$, the orbit $G \cdot y$ is of minimal rank. This
is the content of the previous result.
\end{proof}


\section{Minimal $B$-orbits}

\subsection{Structure of minimal $B$-orbits}

We first prove that the closure of a minimal $B$-orbit is always a
vector space. 

\begin{prop}
\label{prop-min-ev}
Let $L$ and $\g(2)$ be as in Subsection \ref{L-orbit}. Let $x$ such
that $L \cdot x$ is dense in $\g(2)$ and let $Y_0$ be the closure in
$\g(2)$ of the minimal $B$-orbit in $L \cdot x$. Then $Y_0$ is a
vector subspace of $\g(2)$. 
\end{prop}

\begin{proof}
This is easily checked in most cases. In all cases, the vector space $Y_0$ is a sum of weight spaces. For case 1 for example, the closure of the minimal $B$-orbit is the subspace of upper triangular matrices. 
We only discuss case 6 with more details. 

According to \cite[Theorem 4.14]{FR}, any spherical nilpotent orbit can be represented by a sum $y = y_1 + \cdots + y_r$ with $y_i \in \g_{\beta_i}$ (here $\g_\beta$ is the root space associated to the root $\beta$) and such that the root $(\beta_i)_{i \in [1,r]}$ are pairwise orthogonal. Producing the dense $B$-orbit is then easy, for case 6 we get $\beta_1 = (2234321)$, $\beta_2 = (0112221)$ and $\beta_3 = (0000001)$ (here we write $\beta = (abcdefg)$ for $\beta = a \a_1 + b \a_2 + c \a_3 + d \a_4 + e \a_5 + f \a _6 + g \a_7$ where $(\a_1,\a_2,\a_3,\a_4,\a_5,\a_6,\a_7)$ are the simple roots of $E_7$ with notation as in \cite{bourbaki}). To get the minimal $B$-orbit apply the longest element in $W_{E_6}/W_{F_4}$. We get : $\beta_1 = (0112221)$, $\beta_2 = (1112211)$ and $\beta_3 = (1122111)$. The minimal $B$-orbit is therefore contained in the vector space $V$ spanned by the $\g_\beta$ with $\beta \leq \beta_i$ for some $i \in [1,3]$. But an easy checks proves that $[\b,y]$, the image of the tangent action, is exactly $V$ proving the result.
%
\end{proof}


\subsection{Resolution} Applying Corollary \ref{cor-resol2} and the above result on the structure of minimal $B$-orbits, we get, for simply laced groups, a resolution of singularities for any nilpotent orbit of height $2$ .

\begin{cor}
\label{cor-resol3}
Let $G$ be a simple simply laced group, let $x \in \g$
nilpotent of height $2$ and let $Y$ be a $B$-orbit closure.
Then there exists a sequence of minimal parabolics $P_1 , \ldots ,
P_m$ and a vector space $Y_0$ such that  
$$f: P_1 \times^B \ldots \times^B P_m \times^B Y_0 \to Y , \quad [p_1
  , \ldots ,p_m , x] \mapsto p_1 \ldots p_m  \cdot x  ,$$ 
is birational and projective. 
\end{cor}

\begin{remark}
\label{rem-proof}
Note that, since Proposition \ref{prop-min-ev} is true in any case, the above proof works as soon as the spherical conjugacy class is of minimal rank or even with the weaker assumption that the spherical conjugacy class has no type $N$ map (see Proposition \ref{prop-UNT}). We shall use this remark to get more precise results in type $B$ in the next section.
\end{remark}


\section{Singularities outside $G$-orbits}

In this section we prove our first results on the singularities of $B$-orbit closures. 

\subsection{Singularities of nilpotent orbits of height $2$}

In this subsection, we summarise existing regularity results for nilpotent orbits of height $2$. 

\begin{prop}
\label{prop-norm-orb}
Assume $\textrm{char}(\kk) = 0$. Let $G$ be reductive and $x \in \g$ with $\haut(x) = 2$. Then the closure of $G \cdot x$ in $\g$ is normal and has a rational resolution.
\end{prop}

\begin{proof}
Follows directly from \cite[Theorem, page 108]{hesselink} since $\g(2)$ is a completely reducible $P$-representation (the unipotent part of $P$ acts trivially).
\end{proof}

\begin{cor}
Assume $\textrm{char}(\kk) = p > 0$. Let $G$ be reductive and $x \in \g$ with $\haut(x) = 2$. Then the closure of $G \cdot x$ in $\g$ is normal and has a rational resolution for $p$ large enough.
\end{cor}

There exists more precise results for classical groups.

\begin{prop}[Donkin \cite{donkin}, Mehta - van der Kallen \cite{mvdk}, Xiao - Shu \cite{xiaoshu}]
Assume $\textrm{char}(\kk) = p > 0$ and let $G$ be reductive of classical type (\emph{i.e.} $A$, $B$, $C$ or $D$). Assume furthermore $p \neq 2$ for $G$ of type different from $A$. 

Let $x \in \g$ with $\haut(x) = 2$. Then the closure of $G \cdot x$ in $\g$ is normal. In Type $A$ the map $G \times^P \g(2) \to \overline{G \cdot x}$ is a rational resolution.
\end{prop}


\subsection{Singularities of $B$-orbit closures outside $G$-orbits}

\begin{thm}
\label{thm-origin}
Let $G$ be simply laced and $x \in \g$ be nilpotent with $\haut(x) = 2$. The $B$-orbit closure $Y$ of $x$ in $\g$ is normal with rational singularities outside the $G$-orbits in $Y$. 
\end{thm}

\begin{proof}
The closure of $G \cdot x$ is of minimal rank. Hence $Y$ is a multiplicity-free $B$-subvariety of $\overline{G \cdot x}$.  
The result now follows from a result of Brion \cite[Theorem 2]{brion2} (for rational resolutions, see \emph{loc. cit.}, Section 3, last remarks). 
\end{proof}

\begin{remark}
We expect the above $B$-orbit closures to be also normal in $G$-orbits but we were not able to prove this result in general. In the next two sections, we give a proof of this in type $A$ and for nilpotent element of rank $2$. Note however that the fact that we have a resolution with connected fibers implies that the orbit closure $\overline{B \cdot x}$ is homeomorphic to its normalisation: the normalisation map is finite thus proper and bijective.
\end{remark}


\section{Frobenius splitting}

In this section we prove Frobenius splitting results for $G$- and $B$-orbit closures. 

\subsection{Frobenius splitting}

In this subsection, we prove that if $x \in \g$ has height $2$ and if the closure of the nilpotent orbit $G \cdot x$ is normal, then it is Frobenius split. 

\begin{prop}
\label{prop-fs}
Let $\kk$ be of characteristic $p$ a good prime for $G$.  The variety $G \times^P \g(2)$ is $(p-1)D$-Frobenius split for $D$ an ample divisor. Furthermore the splitting compatibly splits the $G$-orbits. 
\end{prop}

\begin{proof}
Recall the construction of the group $G_E$ as well as the parabolic $P_E$ of $G_E$. Recall also that $L$ is the Levi subgroup of $P$ and $P_E$ containing $T$ the maximal torus. Let $B_E = B \cap G_E$ (recall that $B$ is a Borel subgroup of $G$ contained in $P$ and containing $T$). This is a Borel subgroup of $G_E$. Let $B_L = B \cap L$, this is a Borel subgroup of $L$.

Since $p$ is a good prime for $G$, it is also a good prime for $G_E$. In particular, we have a $L$-equivariant isomorphism $\g(2) \simeq R_u(P_E)$. Now we have a $L$-equivariant isomorphism $R_u(P_E) \simeq P_E w_0^E P_E/P_E$ where $L$ acts on $R_u(P_E)$ by conjugation and $w_0^E$ is the longest element in the Weyl group of $G_E$. Since $G_E/P_E$ is $B_E$-canonically Frobenius split (\cite[Theorem 4.1.15]{BK}) compatibly splitting the $P_E^-$-orbits (here $P_E^-$ is the parabolic subgroup opposite to $P_E$ with respect to $T$). Its open subset $P_E w_0^E P_E/P_E$ is therefore $B_E$-canonically split compatibly with its intersection with the $P_E^-$-orbits. It is therefore also $B_L$-canonically split (see \cite[Lemma 4.1.6]{BK}) compatibly with its intersection with the $P_E^-$-orbits. It follows that $R_u(P_E) \simeq \g(2)$ is $B_L$-canonically Frobenius split compatibly with its intersection with the $P_E^-$-orbits which exactly correspond to the $L$-orbits (this is a well known fact, the description of these intersections is for example given by the distance used in \cite{cmp1,cmp2,cmp3} see also \cite{piotr}).

Now since $x$ is of height $2$, the weights appearing in $P$ are bounded above by $2$ so $P$ acts on $\g(2)$ via its Levi factor $L$ which is also contained in $P_E$ so the action of $P$ on $\g(2)$ coincides with the action of $P_E$. This implies that the $B$-action on $\g(2)$ coincides with the $B_E$-action, therefore $\g(2)$ is also $B$-canonically Frobenius split compatibly with the $L$-orbits. 

By \cite[Theorem 4.1.17]{BK}, we obtain that $G \times^B \g(2)$ is $B$-canonically Frobenius split. Furthermore by \cite[Exercice 4.1.4]{BK}, this splitting compatibly splits all Schubert divisors $\overline{B w_0s_\a B} \times^B \g(2)$, with $w_0$ the longest element in the Weyl group of $G$ and $\a$ any simple root. This splitting also compatibly splits the $G$-orbits (the $G$-orbits are obtained by induction from the $L$-orbits in $\g(2)$).

But we have a natural projection $q : G \times^B \g(2) \to G \times^P \g(2)$ obtained from $G/B \to G/P$ by base extension. In particular, we have $q_* \cO_{G \times^B \g(2)} = \cO_{G \times^P \g(2)}$ and we get by \cite[Lemma 1.1.8]{BK} that $G \times^P \g(2)$ is Frobenius split compatibly splitting all Schubert divisors in $G \times^P \g(2)$ and compatibly splitting the $G$-orbits. In particular, if $D_0$ is the sum of all Schubert divisors in $G/P$, then its inverse image $D$ in $G \times^P \g(2)$ via the map $f : G \times^P \g(2) \to G/P$ induced by the first projection, is compatible Frobenius split. Since $D = f^*D_0$, since $D_0$ is ample on $G/P$ and since $f$ is a vector bundle, we get that $D$ is ample on $G \times^P \g(2)$. Now apply \cite[Theorem 1.4.10]{BK} to get the result.
\end{proof}

\begin{cor}
Let $G$ be reductive, $p$ be a good prime for $G$ and $x \in \g$ of height $2$. Assume that the closure of $G \cdot x$ is normal, then the closure of $G \cdot x$ is Frobenius split.
\end{cor}

\begin{proof}
Consider the birational map $\pi : G \times^P \g(2) \to \overline{G \cdot x}$. By normality we have $\pi_* \cO_{G \times^P \g(2)} = \cO_{\overline{G \cdot x}}$. The result follows from \cite[Lemma 1.1.8]{BK}.
\end{proof}

\subsection{Compatibly split $B$-orbits}

For $G$ of type $A$ or for $x$ of rank at most $2$, we prove that the above splitting compatibly splits the $B$-orbit closures. We start with type $A$. In this case $L = {\rm GL}_n(\kk)\times {\rm GL}_n(\kk)$. Let $B_L$ be a Borel subgroup of $L$. The following result is proved in \cite[Proposition 7.1]{HT}.

\begin{prop}
The 
${\rm GL}_n(\kk)$-embedding $M_n(\kk)$ has a $B_L$-canonical splitting compatibly splitting all the $B_L$-orbit closures. 
\end{prop}

\begin{remark}
1. The above splitting is actually the splitting obtained from Proposition \ref{prop-fs}. We will not use this.

2. The compatibly split subvarieties for the above splittings are described in \cite{embeddings}.
The $B_L$-orbit closures in $M_n(\kk)$ are \textit{matrix Schubert varieties}. For such a variety $Z$, 
there exists a Schubert variety $Z'$ in $\textrm{GL}_{2n}(\kk)$ and a smooth and surjective morphism $Z' \to Z$ 
(see \cite{fulton}). In particular, $Z$ has rational singularities, as this holds for Schubert varieties. 
\end{remark}

Let us now consider rank $2$ nilpotent orbits. In this case $G_E = {\rm SO}_{2n+2}$. Let $T$ be a maximal torus. Let $(e_i)_{i \in [1,2n+2]}$ be the basis of $\kk^{2n+2}$ such that the quadratic form $q$ is associated to the bilinear form $b(e_i,e_j) = \delta_{i,2n+3-j}$. Write $(e_i^*)_{i \in [1,2n+2]}$ for the dual basis. Let $B_E$ be a Borel subgroup containing $T$ and $P_E$ be the maximal parabolic subgroup of $G_E$ containing $B_E$ and associated to the simple root $\alpha_1$ (with notation as in Bourbaki \cite{bourbaki}). The unipotent radical of $P_E$ is $V$ and is isomorphic to an open Schubert variety in $G_E/P_E$ (a smooth $2n$-dimensional quadric). Note that we have $V \simeq \kk^{2n}$ 
and the embeding in the quadric is given by $(x_2,\cdots,x_{2n+1}) \mapsto [1,x_2,\cdots,x_{2n+1},-(x_2x_{2n+1} + \cdots + x_n x_{n+1}))]$.

Let $\psi$ be the $B_E$ canonical splitting of $\kk^{2n}$ obtained from the restriction of the $B_E$-canonical splitting $\phi$ on $G_E/P_E$. Since $\phi$ splits the hypersurfaces given by the following equations: $e_1^*$, $e_{2n+2}^*$, and $e_1^*e_{2n+2}^* + \cdots + e_i^*e_{2n+3-i}^*$ for $i \in [2,n+1]$ (see \cite{KLS}). It follows, that $\psi$ splits the hypersurfaces in $\kk^{2n}$ defined by $e_{2n+2}^*$, and for $i \in [2,n+1]$ by the equation $e_1^*e_{2n+2}^* + \cdots + e_i^*e_{2n+3-i}^*$. 

Let $Y_0 \subset \kk^{2n}$ be the subspace defined by $e_{n+1}^* = \cdots = e_{2n+2}^* = 0$. Let $L$ be the Levi subgroup of $P_E$ containing $T$ and $B_L = B_E \cap L$. 

\begin{prop}
The subspace $Y_0$ is $B_L$-stable and compatibly split. 
\end{prop}

\begin{proof}
Indeed, the $B_L$-stability is clear. We check that $Y_0$ can be obtained via taking irreducible components and intersections of the above hypersurfaces. This will prove that it is compatibly split. Consider the intersection of the locus where $e_{2n+2}^*$ and $e_1^*e_{2n+2}^* + e_2^*e_{2n+1}^*$ vanish. It contains as an irreducible component the locus where $e_{2n+2}^*$ and $e_{2n+1}^*$ vanish. By induction the result follows.
\end{proof}

\begin{cor}
\label{cor-FS}
Let $x \in \g$ be nilpotent of height $2$. Assume that $G$ is of type $A$ or that $x$ has rank at most $2$. Then $\overline{G \cdot x}$ is $B$-canonically Frobenius split compatibly with all $B$-orbits.
\end{cor}

\begin{proof}
It follows since any $B$-orbit is obtained via parabolic induction from the closed $B$-orbit $Y_0$, since $Y_0$ is $B$-canonically split ($B$ acts via $B_L$) and by \cite[Theorem 4.17]{BK}. Note that for the case of rank $1$, $Y_0 = \g(2)$ is $B$ canonically split.
\end{proof}


\section{Singularities of $B$-orbit closures}

In this section we prove our results on the singularities of $B$-orbit closures of nilpotent orbits of height $2$ in type $A$ or of rank at most $2$.

\begin{thm}\label{thmRational}
Let $G$ be simple and simply laced, let $x \in \g$ nilpotent of height $2$. Assume that $G$ is of type $A$ or that $x$ is of rank $2$ and let $Y$ be a $B$-orbit closure.

Then $Y$ is normal. 
\end{thm}

\begin{proof} 
We first prove that $Y$ is normal. But $Y$ is Frobenius split (Corollary \ref{cor-FS}) and admits a resolution with connected fibers (Corollary \ref{cor-resol3}). The result follows from \cite[Proposition 1.2.5]{BK}.
%
\end{proof}

\begin{cor}
\label{main-coro}
Let $G$ be a reductive group with simply laced simple factors. Let $x \in \g$ be nilpotent of height $2$, and of rank at most $2$ in all simple factors of type different from $A$, then any $B$-orbit closure in $\overline{G \cdot x}$ is normal.
\end{cor}

\begin{proof}
Follows from the above result and Corollary \ref{cor-red}.
\end{proof}

\begin{remark}
If $\overline{G \cdot x}$ admits a rational resolution, then by general results of M. Brion (see \cite[Section 3, last remarks]{brion2}) it follows that the resolution constructed in Corollary \ref{cor-resol3} is a rational resolution. This holds in particular for $G$ of type $A$.
\end{remark}

\subsection{On the minimal nilpotent orbit} Let $G$ be simple with Lie algebra $\g$. Pick a Cartan subalgebra
$\mathfrak{h}$, $\Phi \subseteq \mathfrak{h}^*$ the set of roots and a set of simple roots $\Delta \subseteq \Phi$.
The set $\Delta$ determines a partial ordering on $\Phi$, and a Borel subgroup $B \subseteq G$. Let $(H_\alpha , X_\rho \mid \alpha \in \Delta , \rho \in \Phi )$ be a Chevalley basis of $\g$. Denote by $\beta \in \Phi^+$ the highest root. Then $\mathcal{O}_{\mathrm{min}} = G \cdot X_\beta$ is the minimal non-trivial nilpotent orbit of $G$.

\medskip

\noindent Denote by $P \supseteq B$ the parabolic group corresponding to $\Delta(\beta) = \{\alpha \in \Delta \mid (\alpha , \beta) =0 \}$, by $\Phi_{\mathrm{lg}} \subseteq \Phi$ the set of long roots, and by $W$ the Weyl group of $G$ which is generated by the simple reflections $s_{\alpha}, \alpha \in \Delta$.  

\begin{lemma}\label{lemmaMinOrbit} The following holds.

\begin{enumerate}
\item Any $B$-orbit closure in $\overline{\mathcal{O}_{\mathrm{min}}}$ admits a rational resolution of singularities.
\item The set of $B$-orbits of $\mathcal{O}_{\mathrm{min}}$ is $\left\{ \mathcal{O}_{\rho} := B \cdot X_{\rho} \mid \rho \in \Phi_{\mathrm{lg}} \right\}$.
\item $\mathcal{O}_{\rho} \subseteq \overline{\mathcal{O}_{\rho'}}$ if and only if $\rho' \prec \rho$.
\end{enumerate}

\end{lemma}

\begin{proof}
As $\ad_{X_\beta}$ has nilpotency order $3$, the orbit $\mathcal{O}_{\mathrm{min}}$ is of height $2$.
Let $H = [X_\beta , X_{-\beta}]$. Denote by $\mathfrak{l}$ the Lie algebra of the Levi factor of $P$. For $\ad_H$- eigenspaces one has $\g (0) = \mathfrak{l}$ and $\g(2)= \g_{\beta}$. Therefore, the multiplication map 
$$G \times^P \g_{\beta} \to \overline{\mathcal{O}_{\mathrm{min}}}, \quad [g, Y] \mapsto g \cdot Y$$ 
is an equivariant resolution and $\mathcal{O}_\mathrm{min}$ is obtained by parabolic induction on the $L$-variety $\g_{\beta} - \{0\}$. The $L$-variety $\g_\beta$ is of minimal rank hence $\overline{\mathcal{O}_\mathrm{min}}$ is of minimal rank (\cite[Lemma 6]{brion}). Let $Y$ be a $B$-orbit closure of $\overline{\mathcal{O}_\mathrm{min}}$. By Remark \ref{rem-proof}, we conclude that $Y$ admits a resolution of singularities with connected fibers. By Corollary \ref{cor-FS} we have that $Y$ is Frobenius split, as $\g_{\beta}$ consists of elements of rank at most $1$. So, assertion ($1$) follows from \cite[Proposition 1.2.5]{BK}. The minimal $B$-orbit of $\mathcal{O}_\mathrm{min}$ is $B \cdot X_{\beta} = \g_\beta - \{ 0 \}$. Applying \cite[Lemma 6]{brion} to $\mathcal{O}_{\mathrm{min}} \simeq G \times^P B\cdot X_{\beta}$ we obtain that the $B$-orbits of $\mathcal{O}_{\mathrm{min}}$ are
$$\left\{ B \cdot X_{w(\beta)} \mid w \in W^P \right\},$$
where $W^P$ is the set of minimal length representatives of $W/W(L)$. Further, by \textit{loc.cit.} the closure order on
the set of $B$-orbits  identifies with the Bruhat order on $W^P$. Assertions ($2$) and ($3$) follow, because by \cite[Theorem 1]{juteau}
the map 
$$W^P \to \Phi_{\mathrm{lg}}, \quad w \mapsto w(\beta)$$
is an anti-isomorphism of partially ordered sets.    
\end{proof}

\begin{remark}
These results are mostly well known. The reader will find more details on the parametrisation by long roots of $B$-orbits and the Bruhat order in the projectivisation of the minimal nilpotent orbit in \cite{CP}.
\end{remark}


\section{Explicit description of $B$-conjugacy class closures}

In this section, we explicitly describe, for classical groups, the nilpotent conjugacy classes which are spherical as well as the $B$-orbit closures. In non simply laced groups, we give examples of non-normal $B$-orbit closures and in type $B$ we precisely describe which conjugacy classes are of minimal rank and deduce more normality results.

\subsection{On classical Lie algebras}

The Lie algebra $\mathfrak{sl}_n$ is the Lie algebra of trace free square $n$ by $n$ matrices. Denoting by $J_n \in \text{GL}_n (\kk )$ the matrix whose anti-diagonal entries equal one and all other entries are zero, we set 
$$K_n = \left(
\begin{array}{cc}
0&J_n\\
-J_n&0
\end{array}
\right)$$ 
and we define the symplectic groups and special orthogonal groups
$$\text{Sp}_{2n} = \left\{ g \in \text{SL}_{2n} \mid g^T K_n g =   K_n \right\} \quad \textrm{ and } \quad
\text{SO}_n = \left\{ g \in \text{SO}_n \mid g^T J_n g = J_n \right\}.$$
The subgroups of upper-triangular (resp. diagonal) matrices form a Borel subgroup (resp. a maximal torus) and will be denoted by $B$ (resp. by $T$).
The corresponding Lie algebras are
$$\mathfrak{sp}_{2n} = \left\{ x \in \mathfrak{sl}_{2n} \mid x^T = K_n x K_n \right\} \quad \textrm{ and } \quad \mathfrak{so}_n = \left\{ x \in \mathfrak{sl}_n \mid x^T = -J_n x J_n \right\}.$$
The Weyl group of $\text{Sp}_{2n}$ and the simple reflections in it we identify with
$$W_{C_n} = \{ \sigma \in S_{2n} \mid \sigma (1)+\sigma (2n)= \ldots = \sigma (n)+ \sigma (n+1)=2n+1 \},$$
$$\Delta_{C_n} = \left\{ c_i := s_i s_{2n-i} \mid i \in [n-1] \right\} \cup \left\{ c_n := s_n \right\}.$$
For $n = 2r+1$, the Weyl group of $\text{SO}_{n}$ and its simple reflections are
$$W_{B_r}= \{ \sigma \in S_{n} \mid \sigma (1) + \sigma(n) = \ldots = \sigma (r) + \sigma (r+2)= n+1 , \sigma(r+1)=r+1 \},$$
$$\Delta_{B_r} = \left\{ b_i := s_i s_{n-i} \mid i \in [r-1] \right\} \cup \left\{b_r := s_r s_{r+1} s_r \right\},$$
and if $n = 2r$ we will use
$$W_{D_r} = \{ \sigma \in S_n \mid \sigma (1) + \sigma (n)= \ldots = \sigma (r) + \sigma (r+1 )= n+1 , \mid \text{neg}(\sigma )\mid \equiv 0 \text{ mod }2 \},$$
$$\Delta_{D_r} = \left\{d_i = s_i s_{n-i} \mid i \in [r-1] \right\} \cup \left\{d_r := s_r s_{r-1}s_{r+1}s_r \right\},$$
where $\text{neg}(\sigma )= \{ i \in [r] \mid \sigma (i) > r \}$.

\subsection{Spherical nilpotent orbits}

In classical Lie algebras one can also describe spherical nilpotent elements directly from their matrices. This was observed by D. Panyushev in \cite{Panyushev} and generalised in any good characteristic in \cite{FR}. The height $2$ spherical nilpotent elements are:
\begin{enumerate} 
\item conjugacy classes of a $2$-nilpotent element in $\mathfrak{sl}_n$, $\mathfrak{sp}_{2n}$ and $\mathfrak{so}_{n}$ ; 
\item conjugacy classes of a $3$-nilpotent element of rank $2$ in $\mathfrak{so}_{n}$.
\end{enumerate}
In the next subsection we shall consider the different cases.

\subsection{The group $G=\SL_n$} 
Let $(e_i)_{i \in [1,n]}$ be the canonical basis
of $\kk^n$ and for $m \leq n$, let $V_m= \scal{e_i \ | \ i \in [1,m]}$
be the span of the first $m$ basis vectors.
Let $\cN_2$ be the set of nilpotent elements of order at most $2$ in $\g = \mathfrak{sl}_n$. 
The conjugacy classes of 2-nilpotent matrices are indexed by the rank. 
Note that the rank $r$ of a nilpotent element of order $2$ in
satisfies $2r \leq n$. Denoting by $\mathbb{O}_r$ the subscheme of
nilpotent elements of order $2$ and rank $r$ we get that the conjugacy classes are $(\mathbb{O}_r)_{2r \leq n}$. 
For $2r \leq n$, let $P^r$ be the parabolic subgroup of $G$
stabilising the flag $V_r \subset V_{n - r}$. We have a natural
morphism
$$p : \mathbb{O}_r \to G/P^r$$
defined by $p(x) = ( \im(x) , \Ker(x) )$. This morphism extends to a
birational transform of the closure of $\mathbb{O}_r$ as follows. Let $X_r$ be
the variety of pairs
$$X_r = \{ (x,(I,K)) \in \g \times G/P^r \ | \ \im(x) \subset I \subset K
\subset \Ker(x) \}.$$
There is a natural $G$-equivariant morphism $p : X_r \to G/P^r$ given by
the second projection as well as a $G$-equivariant morphism $\pi : X_r
\to \g$ given by the first projection. This last morphism is birational
onto the closure of $\mathbb{O}_r$:
$$\overline{\mathbb{O}_r} = \{ x \in \g \ | \ x^2 = 0 \textrm{ and } \rk(x)
\leq r \}.$$
The above morphism $p$ is the induction over $G/P^r$ described in the previous sections. Let
$\A_r$ be the fiber of $p$ over the flag $V_r \subset V_{n-r}$ \emph{i.e.}
$$\A_r = \{ x \in \g \ | \ \im(x) \subset V_r \subset V_{n - r}
\subset \Ker(x) \}.$$
For the adjoint action by conjugation, the parabolic subgroup $P^r$
stabilises $\A_r$.  We may therefore define the contracted product $G
\times^{P^r} \A_r$ and an easy check gives an isomorphism $G
\times^{P^r} \A_r \simeq X_r$. If we identify these two varieties and
write $[g,x]$ for the class of an element $(g,x) \in G \times \A_r$ in
the contracted product, we can describe the morphisms $p$ and $\pi$ as
follows:  
$$p([g,x]) = g \cdot P_r \textrm{ and } \pi([g,x]) = g . x =
gxg^{-1}.$$

\subsubsection{Closure order on $\mathbb{O}_r$}
In this section we use the results of \cite[Lemma 6]{brion1} 
to give a complete description of the $B$-orbits in $\mathbb{O}_r$.

$W= S_n$ is the Weyl group of $G$, considered as a subgroup of $G$. 
We denote by $W^{P^r}$ the set of minimal length representatives in $W$ of the quotient
$W/W_{P^r}$ and by $W_r$ the Weyl group of $\textrm{GL}_r(\kk)$. Let $B_r$
be the Borel subgroup of upper triangular matrices of
$\textrm{GL}_r(\kk)$. Define $x_r \in \A_r$ by
 \begin{equation*}
 x_r (e_i) = \left\{
\begin{array}{rl}
 0 & \text{if } i= 1 , \ldots , n-r\\
 e_{i-(n-r)} & \text{if } i = n-r+1 ,  \ldots , n 
\end{array}\right.
\end{equation*}
Denote the stabilizer of $x_r$ in $G$ by $C_r$. A direct computation shows that 
$C_r$ is the subgroup of $P^r$ consisting of the matrices whose upper-left and lower-right block coincide.
Denote by $\prec$ the Bruhat-Chevalley order on $W$.
The minimal parabolic subgroup generated by $B$ and $s_i$ is denoted by $P_i$. 
Further, we set $W (C_r ) := W \cap C_r$.
\begin{lemma}\label{lemmaOrder}
The following holds:
\begin{enumerate}
\item $\mathbb{O}_r = \coprod_{(\sigma ,w) \in W^{P^r} \times W_r} B \sigma w .x_r$.
\item $B\sigma' w' . x_r \subseteq \overline{B \sigma w. x_r} \subseteq \mathbb{O}_r$
if and only if\\
there exists $\tau \prec \sigma w$ with $\sigma' w' W(C_r ) = \tau W(C_r )$.
\item $\dim B\sigma w . x_r = l(\sigma) + l(w) + \binom{r+1}{2}$. 
\end{enumerate}

\end{lemma}

\begin{proof}
With $\A_r^0 := \{ x \in \A_r \mid \text{rk}(x)= r \}$ 
we have an $B$-equivariant isomorphism
$$G \times^{P^r} \A_r^0 \to \mathbb{O}_r , \quad [g,x] \mapsto g.x .$$  
The identification $\A_r^0 \simeq \mathrm{GL}_r (\kk)$ provides a Bruhat decomposition 
$$\A_r^0 = \coprod_{w \in W_r} B w.x_r .$$
The first claim and the minimality of the $B$-orbit $B.x_r$ follow from \cite[Lemma 6]{brion1}.

Now, choose a reduced expression $\sigma w = s_{i_1} \ldots s_{i_r}$. Using that all covering relations are of type $U$ (again by \cite[Lemma 6]{brion1}) we see that the closure of $B \sigma w . x_r$ in $\mathbb{O}_r$ is given by 
$$P_{i_1} \ldots P_{i_r} . x_r = \bigcup_{\tau \prec \sigma w} B \tau . x_r .$$
The second claim follows, because $\tau' .x_r = \tau . x_r$ if and only if $\tau' W(C_r )= \tau W(C_r)$.

For the third claim, note that $l(\sigma w) = l(\sigma )+l(w)$ since $w \in W_r \subseteq W_{P^r}$. Therewith
\begin{multline*}
\dim B \sigma w . x_r = \dim P_{i_1}\ldots P_{i_r} . x_r =l(\sigma w) + \dim (B. x_r)\\
=l(\sigma w ) + \dim (B_r ) = l(\sigma) + l(w) + \binom{r+1}{2}
\end{multline*}
and the last claim follows.
\end{proof}

\begin{remark} The results of \cite[Lemma 6]{brion1} give a description of the $B$-orbits in $X= G \times^{P^r} \A_r$ as well:
They are indexed by $W^{P^r} \times \mathcal{P}(r)$, where $\mathcal{P}(r)$ is the set of partial permutation matrices
in $M_r (\kk)$.
\end{remark}
\begin{remark} Borel orbits in 2-nilpotent matrices have been investigated before in \cite{Boos-Reineke}. In order to compare our results with these previous results, denote for $(i,j) \in [n]^2$ by $E_{i,j} \in \g$ the corresponding elementary matrix. Then
$$\sigma w . x_r =  \sigma w x_r (\sigma w)^{-1}=\sum_{j=1}^r E_{\sigma w (j) , \sigma (n-r+j)}.$$
This is the $2$-\textit{nilpotent matrix associated to the oriented link pattern on $r$ arcs}
$$(\sigma(n-r+1) , \sigma w(1)), \ldots , (\sigma(n),\sigma w(r)), $$
as in \cite{Boos-Reineke}. It was there shown that oriented link patterns parametrize $B$-orbits in $\overline{\mathbb{O}_r}$ by using representation theory of quivers. The closure order on $B \backslash \overline{\mathbb{O}_r}$ (and not only on $B \backslash \mathbb{O}_r$ as in Lemma \ref{lemmaOrder} $(2)$) was determined in \cite{Boos-Reineke} as well. The term \textit{oriented link pattern} refers to an extension of the term \textit{link pattern} which appeared first in physical literature (cf. \cite{DFZJ}). It was shown in \cite{Melnikov} that link patterns parametrize $B$-orbits of $2$-nilpotent upper-triangular matrices. 
 
\end{remark}

\begin{example}
For $(n,r)=(4,2)$ we have $W_r =  \scal{s_1} $,
$W( C_r ) = \scal{s_1,s_3}$ and 
$$W^{P^r} = \left\{ \sigma \in W \mid \sigma (1) < \sigma (2), \sigma(3)<\sigma(4) \right\} = \left\{ \text{id} , s_2 ,s_1 s_2 , s_3 s_2 , s_1 s_3 s_2 , s_2 s_1 s_3 s_2 \right\}.$$ 
Using Lemma \ref{lemmaOrder} we obtain

\begin{center}
\setlength{\unitlength}{2569sp}%
\begingroup\makeatletter\ifx\SetFigFont\undefined%
\gdef\SetFigFont#1#2#3#4#5{%
  \reset@font\fontsize{#1}{#2pt}%
  \fontfamily{#3}\fontseries{#4}\fontshape{#5}%
  \selectfont}%
\fi\endgroup%
\begin{picture}(4977,6178)(2026,-6611)
\thinlines
{\color[rgb]{0,0,0}\put(5056,-6271){\line( 1, 1){810}}
}%
{\color[rgb]{0,0,0}\put(4471,-6316){\line(-1, 1){825}}
}%
{\color[rgb]{0,0,0}\put(3436,-5026){\line(-1, 1){915}}
}%
{\color[rgb]{0,0,0}\put(2491,-3661){\line( 0, 1){810}}
}%
{\color[rgb]{0,0,0}\put(2686,-2401){\line( 1, 1){720}}
}%
{\color[rgb]{0,0,0}\put(3721,-1246){\line( 1, 1){645}}
}%
{\color[rgb]{0,0,0}\put(6061,-5056){\line( 1, 1){915}}
}%
{\color[rgb]{0,0,0}\put(6991,-3736){\line( 0, 1){960}}
}%
{\color[rgb]{0,0,0}\put(6946,-2416){\line(-1, 1){750}}
}%
{\color[rgb]{0,0,0}\put(3661,-1711){\line( 1,-1){765}}
}%
{\color[rgb]{0,0,0}\put(3991,-1576){\line( 5,-2){2438.793}}
}%
{\color[rgb]{0,0,0}\put(6721,-3766){\line(-4, 1){3924.706}}
}%
{\color[rgb]{0,0,0}\put(4861,-3736){\line( 2, 1){1788}}
}%
{\color[rgb]{0,0,0}\put(4771,-2776){\line( 0,-1){990}}
}%
{\color[rgb]{0,0,0}\put(4516,-2806){\line(-2,-1){1758}}
}%
{\color[rgb]{0,0,0}\put(4966,-4171){\line( 1,-1){885}}
}%
{\color[rgb]{0,0,0}\put(3061,-4006){\line( 5,-2){2550}}
}%
{\color[rgb]{0,0,0}\put(4591,-4156){\line(-1,-1){930}}
}%
{\color[rgb]{0,0,0}\put(4321,-5236){\line( 2, 1){2250}}
}%
{\color[rgb]{0,0,0}\put(4996,-2461){\line( 1, 1){780}}
}%
{\color[rgb]{0,0,0}\put(3031,-2461){\line( 3, 1){2551.500}}
}%
{\color[rgb]{0,0,0}\put(5746,-1276){\line(-1, 1){675}}
}%
\put(3181,-1531){\makebox(0,0)[lb]{\smash{{\SetFigFont{6}{7.2}{\familydefault}{\mddefault}{\updefault}{\color[rgb]{0,0,0}($s_2 ,\text{id}$)}%
}}}}
\put(6526,-2671){\makebox(0,0)[lb]{\smash{{\SetFigFont{6}{7.2}{\familydefault}{\mddefault}{\updefault}{\color[rgb]{0,0,0}($s_2 ,s_1$)}%
}}}}
\put(4321,-2671){\makebox(0,0)[lb]{\smash{{\SetFigFont{6}{7.2}{\familydefault}{\mddefault}{\updefault}{\color[rgb]{0,0,0}($s_3 s_2 ,\text{id}$)}%
}}}}
\put(2206,-2656){\makebox(0,0)[lb]{\smash{{\SetFigFont{6}{7.2}{\familydefault}{\mddefault}{\updefault}{\color[rgb]{0,0,0}($s_1 s_2 ,\text{id}$)}%
}}}}
\put(4306,-3991){\makebox(0,0)[lb]{\smash{{\SetFigFont{6}{7.2}{\familydefault}{\mddefault}{\updefault}{\color[rgb]{0,0,0}($s_3 s_2 ,s_1$)}%
}}}}
\put(6541,-3991){\makebox(0,0)[lb]{\smash{{\SetFigFont{6}{7.2}{\familydefault}{\mddefault}{\updefault}{\color[rgb]{0,0,0}($s_1 s_2 ,s_1$)}%
}}}}
\put(5521,-5281){\makebox(0,0)[lb]{\smash{{\SetFigFont{6}{7.2}{\familydefault}{\mddefault}{\updefault}{\color[rgb]{0,0,0}($s_1 s_3 s_2 ,s_1$)}%
}}}}
\put(5566,-1501){\makebox(0,0)[lb]{\smash{{\SetFigFont{6}{7.2}{\familydefault}{\mddefault}{\updefault}{\color[rgb]{0,0,0}($\text{id},s_1$)}%
}}}}
\put(2041,-3961){\makebox(0,0)[lb]{\smash{{\SetFigFont{6}{7.2}{\familydefault}{\mddefault}{\updefault}{\color[rgb]{0,0,0}($s_1 s_3 s_2 ,\text{id}$)}%
}}}}
\put(3136,-5326){\makebox(0,0)[lb]{\smash{{\SetFigFont{6}{7.2}{\familydefault}{\mddefault}{\updefault}{\color[rgb]{0,0,0}($s_2 s_1 s_3 s_2 , \text{id}$)}%
}}}}
\put(4156,-6556){\makebox(0,0)[lb]{\smash{{\SetFigFont{6}{7.2}{\familydefault}{\mddefault}{\updefault}{\color[rgb]{0,0,0}($s_2 s_1 s_3 s_2 ,s_1$)}%
}}}}
\put(4486,-556){\makebox(0,0)[lb]{\smash{{\SetFigFont{6}{7.2}{\familydefault}{\mddefault}{\updefault}{\color[rgb]{0,0,0}($\text{id,id}$)}%
}}}}
\end{picture}%
\\
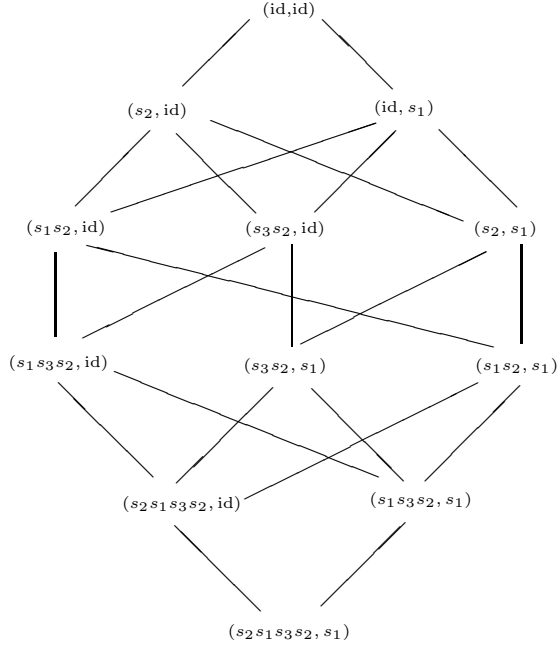
\textbf{Figure 1. Type $A_{n-1}$ closure graph for $(n,r)=(4,2)$ .} 
\end{center}
\end{example}

\subsubsection{$B$-orbit closures as sets of linear maps.}
Recall that $V_i \subseteq \kk^n$ denotes the coordinate subspace generated by $e_1 , \ldots ,e_i$.
For $0 \leq 2r \leq n$ and $(\sigma ,w ) \in W^{P^r} \times W_r$ we consider $Z(\sigma , w) = \overline{B \sigma w . x_r} \subseteq \cN_2$. All $B$-conjugacy class closures in $\cN_2$ are of this form. In order to describe 
$Z(\sigma ,w)$ explicitly, we define 
$$r(i,j ,  x) := \dim \left( x(V_i )+V_j \right) ,\text{ for } (i,j) \in [n] \times [n] \cup \{ 0\},$$
where we have set $V_0 := \left\{ 0 \right\}$.
\begin{lemma}\label{lemmaEquations} As a set of linear maps,
$$Z(\sigma , w) = \left\{ x \in \cN_2 \mid r(i,j , x) \leq r(i,j , \sigma w . x_r ), \text{ for all }  (i,j) \in [n] \times [n] \cup \{ 0\} \right\}.$$
\end{lemma}

\begin{proof} This result appears in \cite{Rothbach}. We give a shorter proof. The space $\g = M_n (\kk)$ is the disjoint union of $(B,B)$-double cosets of partial permutation matrices. 
Further, for $x,y \in \g$ one has $BxB = ByB$ if and only if
$$r (i,j , x) = r(i,j ,y ), \text{ for all } (i,j) \in [n] \times [n] \cup \{ 0 \}.$$
From Lemma \ref{lemmaOrder} we derive that $B\sigma w.x_r = \mathbb{O}_r \cap B (\sigma w .x_r) B$. Consequently, 
$$
B \sigma w .x_r =\left\{ x \in \mathbb{O}_r \mid r(i,j , x) = r(i,j , \sigma w . x_r ), \text{ for all }  (i,j) \in [n] \times [n] \cup \{ 0\} \right\}. $$
Taking the closure in $\cN_2$ we obtain the wanted description of $Z(\sigma ,w)$.
\end{proof}

\begin{example}\label{exampleNonGorenstein}
Let $(n,r)=(3,1)$. Then $W_r = \{ 1 \}$ and $W^{P_r} = W$. Take $(\sigma , w) = ( s_1 s_2 , 1 )$. Then $Z(s_1 s_2 , 1)$ is 3-dimensional. Using Lemma \ref{lemmaEquations} we see it is isomorphic to
$$Z= \left\{ \left( \begin{array}{cc}
A & v \\
0 & 0 \\
\end{array}
\right) \ | \ (A , v ) \in M_2 (\kk) \times \kk^2,\quad A^2 =  0 \textrm{ and } Av = 0  \right\}.$$
This is a toric variety: The open torus is 
$$T_Z := Z \cap (M_2 (\kk^* ) \times (\kk^* )^2 ) = \left\{ \left( \left(
\begin{array}{cc}
t_1 & -t_2\\
\frac{t_1^2}{t_2}&-t_1 
\end{array}\right) , \binom{t_3}{\frac{t_1 t_3}{t_2}} \right) \mid t_1 ,t_2 ,t_3 \in \kk^* \right\}.$$
By considering $1$-parameter subgroups of $T_Z$ we see that $Z$ is the toric variety associated with the cone
$$\mathbb{Q}_{\geq 0} \left(
\begin{array}{c}
0\\
0\\
1 
\end{array}\right) + \mathbb{Q}_{\geq 0} \left(
\begin{array}{c}
1\\
0\\
0 
\end{array}\right) + \mathbb{Q}_{\geq 0} \left(
\begin{array}{c}
1\\
1\\
0 
\end{array}\right) + \mathbb{Q}_{\geq 0} \left(
\begin{array}{c}
1\\
2\\
1 
\end{array}\right). $$ 
We conclude that $Z$ is not Gorenstein, as the generators of this cone are not contained in an affine hyperplane in $\Q^3$.
\end{example}

\subsection{The group $\Sp_{2n}$}

\subsubsection{Symplectic $2$-nilpotent elements}
A $2$-nilpotent element in $\mathfrak{sp}_{2n}$ is of rank $1 \leq r \leq n$. These elements form a $\text{Sp}_{2n}$-conjugacy class we denote by $\mathcal{C}_r$. 
Further we will be dealing with 
$$\mathcal{S}_r = \left\{ x \in \mathcal{C}_r \mid V_r = \text{Im}(x) \subseteq \text{Ker}(x)=V_{2n -r} \right\},$$
and $P^r \subseteq \text{Sp}_{2n}$ the parabolic subgroup given by elements stabilizing the vector subspace $V_r$, which acts on $\mathcal{S}_r$.
Denoting by $L^r \subseteq P^r$ the Levi factor, we see that the $P^r$-action on $\mathcal{S}_r$ factors through $P^r \to L^r$ and that the multiplication map
$$\text{Sp}_{2n} \times^{P^r} \mathcal{S}_r \to \mathcal{C}_r , \quad [g,x] \mapsto g.x,$$
is an isomorphism. So $\mathcal{C}_r$ is obtained by parabolic induction on the $L^r$-variety $\mathcal{S}_r$.\\  
Denote by $S_r^{(2)}$ the set of involutions in $S_r$ and by $R_r := \{ \tau \in S_r  \mid J_r \tau \in S_r^{(2)} \}$. Define
$$W^r := \{ \sigma \in W_{C_n} \mid \sigma (1) < \ldots < \sigma (r),\quad  \sigma (r+1) <\ldots < \sigma(2n-r)\},$$
forming the set of minimal length representatives of $W_{C_n}/ W(L^{r})$. Finally, for a permutation $\sigma \in S_{2n}$ we set $\text{pos}(\sigma)= \{ i \in [n] \mid \sigma (i)<n \}$ and $\text{neg}(\sigma) = [n]- \text{pos}(\sigma)$.

\begin{lemma}\label{lemmaBOrbitsTypeC}
A set of representatives for the $B$-conjugacy classes in $\mathcal{C}_r$ is given by the elements
$$ x(\sigma , \tau):=  \sum_{\tau(i) \in \text{pos}(\sigma)} E_{\sigma \tau (i), \sigma(2n-r+i)} - \sum_{\tau(j) \in \text{neg}(\sigma)} E_{\sigma \tau (j), \sigma(2n-r+j)},$$
with $(\sigma , \tau ) \in W^r \times R_r$. 
\end{lemma}

\begin{proof}
An element $\sigma \in W_{C_n} = N(T)/T$ is represented by the element $\dot\sigma \in N(T)$ given by
$$\dot\sigma (e_i) = \left\{
\begin{array}{rl}
e_{\sigma (i)} & \text{, if }i \in \text{pos}(\sigma) \cup \{n+1 , \ldots , 2n \} \\
-e_{\sigma(i)} & \text{, if }i \in \text{neg}(\sigma) .
\end{array} \right\},$$
and for any $\tau \in R_r$ one computes that
$$x(\sigma , \tau )= \dot\sigma . \left( 
\begin{array}{ccc}
0&0&\tau \\
0&0&0 \\
0&0&0
\end{array}\right).$$
The morphism which maps an element of $\mathfrak{sp}_{2n}$ to its upper-right-hand $r \times r$-block induces an equivariant isomorphism between the $L^r$-variety $\mathcal{S}_r$ and the $\text{GL}_r$-variety $Y_r =\{ h \in \text{GL}_r \mid h^T = J_r h J_r \}$, where $\text{GL}_r$ acts via
$$g.h = gh J_r g^T J_r , \quad g \in \text{GL}_r , h \in Y_r .$$
By \cite{RS1} $B_r$-orbits in the symmetric space $Y_r \simeq \text{GL}_r / O_r$ are indexed by $R_r$. As we have seen, $\mathcal{C}_r$ is obtained by parabolic induction on the $L^r$-variety $\mathcal{S}_r$. Hence the claim follows from \cite[Lemma 6]{brion1}.
\end{proof}

\begin{lemma}\label{lemmaEquationsTypeC}
Let $(\sigma , \tau ) \in W^r \times R_r$. Then
$$\overline{B \cdot x(\sigma , \tau )} = \left\{ x \in \overline{\mathcal{C}_r }  \mid r(i,j , x) \leq r(i,j , x(\sigma , \tau ) ), \text{ for all }  (i,j) \in [2n] \times [2n] \cup \{ 0\} \right\}.$$
\end{lemma}

\begin{proof}
The claim will follow from Lemma \ref{lemmaEquations} once we have seen that  
$B \cdot x(\sigma , \tau ) = B_{2n} \cdot x(\sigma , \tau ) \cap \mathfrak{sp}_{2n}$, 
where $B_{2n}$ denotes the Borel subgroup of $\SL_{2n}$ containing $B$. Clearly the LHS is contained in the RHS. To see the other inclusion, use that
$x(\sigma' , \tau' ) = x(\sigma , \tau)$ if and only if $r(i,j , x(\sigma' , \tau' )) = r(i,j , x(\sigma , \tau ))$, for all $i,j$. 
\end{proof}

\begin{lemma}\label{lmmWeakOrderTypeC}
Let $\mathcal{O}(\sigma , \tau)= B.x(\sigma , \tau )$, for a pair $(\sigma , \tau ) \in W^{r} \times R_r$ and $P_i \supseteq B$ a minimal parabolic corresponding to a simple refelction $c_i \in \Delta_{C_n}$. Assume that $P_i$ raises $\mathcal{O}(\sigma , \tau)$.
Then three cases occur:
\begin{enumerate}
\item  $P_i \mathcal{O}(\sigma , \tau) = \mathcal{O}(c_i \sigma , \tau) \coprod \mathcal{O}(\sigma , \tau)$ iff $c_i \sigma \in W^{r}$.\\
This gives a covering relation of type $U$.
\item  $P_i \mathcal{O}(\sigma , \tau) = \mathcal{O}(\sigma , s_i \tau J_r s_i J_r ) \coprod \mathcal{O}(\sigma, \tau)$ iff $i \in [r-1]$ {and} $s_i \tau J_r s_i J_r \neq \tau$.\\
This gives a covering relation of type $U$.
\item $P_i \mathcal{O}(\sigma , \tau) = \mathcal{O}(\sigma , s_i \tau ) \coprod \mathcal{O}(\sigma , \tau)$ iff $i \in [r-1]$ {and} $s_i \tau J_r s_i J_r = \tau$.\\
This gives a covering relation of type $N$.
\end{enumerate}

\end{lemma}

\begin{proof}
Follows again from \cite[Lemma 6]{brion1} and B. J. Wyser's  and Richardson-Springer's results \cite{RS1, Wyser} for the weak order on $B_r \backslash Y_r$ .
\end{proof}

A consequence of Zariski's main theorem is the following

\begin{lemma}\label{lmmSphericalZariski}\cite[Corollary 4.4.4]{survey}
Let $X$ be a $G$-spherical variety, $Y$ a $B$-orbit and $P_1 ,P_2 \supseteq B$ minimal parabolics. Assume that $P_1$ raises $Y$ to $Y_1$ with Type $U$ or $T$, $P_2$ raises $Y$ to $Y_2$ with Type $N$ and $P_2$ raises $Y_1$ to $Y_3$ with Type $U$ or $T$.
Then, the $B$-orbit closure $\overline{Y_3}$ is not normal along $\overline{Y_2}$.
\end{lemma}

\begin{remark}\label{rmrPinPhD}
S. Pin in \cite{PinPhD} has investigated the singularities in $B_r$-orbit closures 
in the symmetric space $\text{GL}_r / \text{O}_r \simeq Y_r$. He gave a criterion for them to be regular in codimension $1$. Furthermore he found non-normal orbit closures which are regular in codimension $1$, hence examples of orbit closures which are neither normal nor Cohen-Macaulay.
\end{remark}

Likewise, in $\mathcal{C}_r$ there occur $B$-conjugacy closures with similar properties, as the next example shows.

\begin{example}\label{exampleNon-normal/Non-CM}
Let $(n,r)=(2,2)$. Then $R_2 =S_2$ and $W^2 = \left\{1, c_2 , c_1 c_2 , c_2 c_1 c_2 \right\}$. Denote by $P_i \supseteq B$ the minimal parabolic corresponding to the simple reflection $c_i$. Using Lemma \ref{lmmWeakOrderTypeC}, the covering relations are the following:
$$
\xymatrix{
& (1,1) \ar@{-}[dl]_{P_2} \ar@{=}[dr]^{P_1} & \\
(c_2 , 1 ) \ar@{-}[d]_{P_1}&  & (1, s_1 ) \ar@{-}[d]^{P_2}  \\
(c_1 c_2 , 1 ) \ar@{-}[d]_{P_2} &  & (c_2 , s_1 ) \ar@{-}[d]^{P_1}  \\
(c_2 c_1 c_2 , 1 ) \ar@{=}[dr]_{P_1} &  & (c_1 c_2 , s_1 ) \ar@{-}[dl]^{P_2}  \\
& (c_2 c_1 c_2 , s_1 ) & \\
 },
$$
\begin{center}
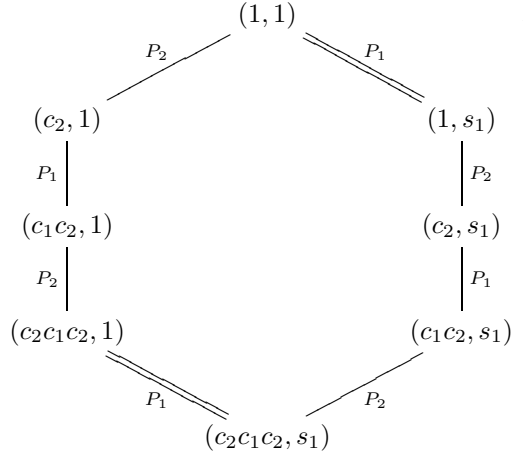
\textbf{Figure 2. Type $C_n$ weak order for $(n,r)=(2,2)$.}
\end{center}

where a double edge means that the covering relation is of type $N$. With Lemma \ref{lmmSphericalZariski} we find two non-normal $B$-conjugacy class closures:
$$\overline{\mathcal{O}( c_1 c_2 , 1)}= \left\{ \left(
\begin{array}{cc}
x&y\\
0&-J_2 x^t J_2 
\end{array}\right) \mid x^2 = 0 , xy=J_2 y^t J_2 \right\}$$
$$\textrm{	and } \quad \overline{\mathcal{O}( c_2 c_1 c_2 , 1)}= \left\{ z \mid z_{4,1}=0 \right\} \subseteq \mathcal{N}_2 (\mathfrak{sp}_4 ).$$
The non-normal and the singular loci agree and are in both cases given by $\overline{\mathcal{S}_2 } = \overline{\mathcal{O}(1, s_1 )} \simeq \mathbb{A}^3$.
So, $\overline{\mathcal{O}( c_1 c_2 , 1)}$ is singular in codimension $1$, whereas $\overline{\mathcal{O}( c_2 c_1 c_2 , 1)}$ is not. By Serre's criterion for normality, we may conclude that the latter is not Cohen-Macaulay.
\end{example}

\subsection{The group $\SO_{n}$}

\subsubsection{Orthogonal $2$-nilpotent elements} A $2$-nilpotent element in an orthogonal Lie algebra has even rank, hence we are dealing with $2$-nilpotent elements of rank $2s$ where $4s \leq n$. The corresponding nilpotent orbits are denoted by $\mathcal{B}_{2s}$ with the exception of the \textit{very even case} $n=4s$; then, $2$-nilpotent elements of rank $2s$ form a $\text{O}_n$-conjugacy class, which is the union of two $\text{SO}_n$-conjugacy classes. We treat this case later on.\\

We will use $\mathcal{K}_{2s} = \{ x \in \mathcal{B}_{2s} \mid \text{Im}(x) = V_{2s} \subseteq V_{n-2s} = \text{Ker}(x) \}$ and $P^s$, the parabolic subgroup of $\text{SO}_n$ given as the stabilizer of the partial flag $V_{2s} \subseteq V_{n-2s}$. The $P^s$-action on $\mathbb{K}_s$ factors through the morphism $P^s \to L^s$ ($L^s$ denoting the Levi factor of $P^s$) and we have that $\mathcal{B}_{2s}$ is obtained by parabolic induction on the $L^s$-variety $\mathcal{K}_{2s}$, since the multiplication map
$$\text{SO}_n \times^{P^s} \mathcal{K}_{2s} \to \mathcal{B}_{2s} , \quad [p,x] \mapsto p.x$$
is an isomorphism. Denote by $F_s$ the set of permutations $\tau \in S_{2s}$ such that $J_{2s}\tau$ is a fixed-point free involution, and by $W^s$ the set of minimal length representatives of $W(L^s )$ right cosets in $W(\text{SO}_n )$.
For a permutation $\rho \in S_{f}$ we will regard the set $\text{def}(\rho ) = \{i \in [f] \mid \rho (i) + i > f+1 \}$.

\begin{lemma}\label{lmmBOrbitsOrthogonal} The set consisting of the elements
$$y(\sigma ,  \tau ) := - \sum_{i \in \text{def}(\tau ) } E_{\sigma (n+1-\tau(i)) , \sigma (n-2s+i)} + \sum_{i \notin \text{def}(\tau)} E_{\sigma (n+1-i) , \sigma(n-2s + \tau (i))},$$
with $(\sigma , \tau) \in W^s \times F_s$, parametrizes Borel conjugacy classes in $\mathcal{B}_{2s}$.
\end{lemma}

\begin{proof}
This is similar to the proof of Lemma \ref{lemmaBOrbitsTypeC}. The morphism which maps an element of $\mathfrak{so}_{n}$ to its upper-right-hand $(2s\times 2s)$-block induces an equivariant isomorphism $\varphi : \mathcal{K}_{2s} \to Z_{2s}$ between the $L^s$-variety $\mathcal{K}_{2s}$ and the $\text{GL}_{2s}$-variety $Z_{2s} =\{ h \in \text{GL}_{2s} \mid h^T = - J_{2s} h J_{2s} \}$, where $\text{GL}_{2s}$ acts via
$$g.h = gh J_{2s} g^T J_{2s} , \quad g \in \text{GL}_{2s} , h \in Z_{2s}.$$
By \cite{RS1, Wyser} $B_{2s}$-orbits in the symmetric space $Y_{2s}$ are indexed by $F_s$ as follows: An element $\tau \in F_s$ is identified with the $B_{2s}$-orbit through
$$\tilde{\tau} :=  - \sum_{i \in \text{def}(\tau ) } E_{2s+1-\tau(i) , i} + \sum_{i \notin \text{def}(\tau)} E_{2s+1-i , \tau (i)} \in Z_{2s}.$$
If $n = 2r$ is even, the Weyl group $W_{D_r}$ consists of even permutations, hence $\sigma \in W^s$ is represented by the ordinary permutation matrix $\dot\sigma \in N(T)$. If $n = 2r+1$ and $\sigma \in W^s$ is not an even permutation, we choose the representative given by
$$\dot\sigma ( e_i ) = \left\{
\begin{array}{rl}
e_{\sigma (i)} & \text{, if } i \neq r+1 \\
-e_{r+1} & \text{, if }i = r+1
\end{array}
\right\}.$$
In any case we compute that $\dot\sigma . \varphi^{-1}(\tilde{\tau}) = y(\sigma , \tau )$. Now, the claim follows from \cite[Lemma 6]{brion1}, as $\mathcal{B}_{2s}$ is obtained from $Z_{2s}$ by parabolic induction.   
\end{proof}

It remains to investigate the very even $2$-nilpotent case, i.e. the case where $n = 4s$. Then $\mathcal{V}= \{ x \in \mathfrak{so}_n \mid \text{Im}(x)=\text{Ker}(x) \}$
is a $\text{O}_n$-conjugacy class which consists of two $\text{SO}_n$-conjugacy classes. Denote by $m \in \text{O}_n$ the 
permutation matrix switching $2s$ with $2s+1$. Then $\text{SO}_n$ and $m \text{SO}_n$ are the connected components of $\text{O}_n$. We conclude that $\mathcal{V}$ is the union of the two $\text{SO}_n$-conjugacy classes 
$\mathcal{B}_{2s} \simeq \text{SO}_n \times^{P^s} \mathcal{K}_{2s}$ and $\mathcal{B}_{2s}^{'} = m. \mathcal{B}_{2s}$. Now $m B m= B$, whence with Lemma \ref{lmmBOrbitsOrthogonal} we see the elements $m. y(\sigma , \tau)$, where $(\sigma , \tau ) \in W^s \times F_s$, parametrize $B$-conjugacy classes in $\mathcal{B}_{2s}^{'}$.

\begin{cor}\label{corBDMinimalRank}
$\mathcal{N}_2 (\mathfrak{so}_n ) := \left\{ x \in \mathfrak{so}_n \mid x^2 = 0 \right\}$ is a $\text{SO}_n$-spherical variety of minimal rank. 
\end{cor}

\begin{proof}
For $n$ not a multiple of $4$, the non-zero $\text{SO}_n$-orbits in $\mathcal{N}_2 (\mathfrak{so}_n )$ are  $\mathcal{B}_{2s}$, where $1 \leq 4s < n$. These are spherical varieties of minimal rank as they are obtained by parabolic induction on the $L^s$-spherical varieties $Z_{2s}$, which are of minimal rank (cf. \cite{RS1}).
For $n = 4s$, in addition one has the $\text{SO}_n$-orbit $\mathcal{B}_{2s}^{'}$ which is $B$-equivariantly isomorphic to $\mathcal{B}_{2s}$. Hence $\mathcal{N}_2 (\mathfrak{so}_n )$ has the asserted property.  
\end{proof}

\begin{cor}\label{thmRationalOrthogonal}
Any $B$-orbit closure $Y$ of a $2$-nilpotent orbit in $\mathfrak{so}_n $ is normal and has a rational resolution of singularities outside of $\text{SO}_n$-orbits in $Y$. 
\end{cor}

\begin{proof} 
Follows from the above statement and Remark \ref{rem-proof}.
\end{proof}

\begin{remark}
We believe that the above statement is also true for the $\text{SO}_n$-orbits in $Y$.
\end{remark}

\subsubsection{$3$-nilpotent element}\label{3Nilp}
We are dealing with the orbits 
$$\mathcal{O}_n = \{ z \in \mathfrak{so}_n \mid \text{rg}(z)=2,\quad \text{rg}(z^2)=1 ,\quad z^3 = 0 \}.$$
We assume that $n \geq 4$, as we already treated the case $\mathfrak{so}_3 \simeq \mathfrak{sl}_2$. We fix $x:= E_{1,2} + E_{1,n-1} - E_{2,n} - E_{n-1,n} \in\mathcal{O}_n$. A $\mathfrak{sl}_2$-triple is then given by
$\{x,y,h\}$, where $y:= E_{2,1} + E_{n-1,1} - E_{n,2} - E_{n,n-1}$ and $h: = [x,y] = 2 E_{1,1} - 2 E_{n,n}$. One computes that the $\ad_h$-eigenspace decomposition is 
$$\mathfrak{so}_n (0) = \left\{ b \in \mathfrak{so}_n \mid b_{1,2}= \ldots = b_{1,n-1} = b_{2,1}= \ldots = b_{n-1,1}=0 \right\} \simeq \mathfrak{so}_{n-2} \oplus \kk,$$
$$\mathfrak{so}_n (2) = \bigoplus_{j=2}^{n-1} k \cdot (E_{1,j} - E_{j,n}), \quad
\mathfrak{so}_n (-2) = \bigoplus_{j=2}^{n-1} k \cdot (E_{j,1}-E_{n,n+1-j}).$$
Define $F_n := \mathfrak{so}_n (2) \cap \mathcal{O}_n$. Denote by $P \supseteq B$ the maximal parabolic subgroup corresponding to the first simple reflection in $\Delta_{\text{SO}_n }$ if $n>4$ and $P=B$ if $n=4$. The Levi factor $L \simeq k^* \times \text{SO}_{n-2}$ is the connected subgroup of $\text{SO}_n$ with Lie algebra $\mathfrak{so}_n (0)$. Further, $P$ acts on $F_n$ by conjugation and this action factors through the map $P \to L$. One may then consider the map
$$\text{SO}_n \times^P F_n \to \mathcal{O}_n , \quad [g,z] \mapsto g.z.$$
The centralizer of $x$ is contained in $P$. Hence this map is an isomorphism. Therefore, $\mathcal{O}_n$ is obtained by parabolic induction on $F_n$, which identifies $L$-equivariantly with
$$U_{n-2} = \left\{ x \in k^{n-2} \mid x^T J_{n-2} x \neq 0 \right\},$$
upon which $L$ acts transitively via 
$$(\lambda , g).x = \lambda g(x),  \quad (\lambda , g ) \in L, x \in U_{n-2}.$$
The $B \cap L$-orbits in $U_{n-2}$ are given by
$$O_i := \left\{ \left(
\begin{array}{c}
x_1 \\
\vdots\\
x_{n-2} 
\end{array} \right)
 \in U_{n-2} \mid x_{n-1-i} \neq 0 , x_{n-i}= \ldots = x_{n-2} = 0 \right\},$$
where $i \in [r-1]$ for $n=2r$ even and $i \in [r]$ for $n=2r+1$ odd. Further, one computes that
$$\text{rank}(O_i )= 
\left\{
\begin{array}{rl}
1 & \text{ if } i=r\\
2 & \text{ if } i<r   
  \end{array}\right\}.$$
We derive that $B \cap L$-orbits in $F_n$ are parametrized by the set 
$$I_n = \left\{ f_i := E_{1,i+1} + E_{1,n-i} - E_{n-i,n} - E_{i+1,n} \mid i \in [r-1] \right\},$$
if $n=2r$ is even, and $I_n = \left\{ f_1 , \ldots , f_{r-1}, f_r := E_{1,r+1} - E_{r+1,n} \right\}$, if $n=2r+1$ is odd.
We have obtained 
\begin{lemma}\label{lmmBOrbits3}
The $B$-orbits in $\mathcal{O}_n$ are parametrized by $W^P \times I_n$. For $(\sigma , f_i ) \in W^P \times I_n$ the corresponding $B$-orbit is of rank $2$ if and only if $i<r$. For $i=r$ the rank is equal to $1$. In particular, $\mathcal{O}_n$ is of minimal rank if and only if $n$ is even. 
\end{lemma}

\begin{cor}\label{corNormalB_2r}
Any $B$-orbit closure of $\overline{\mathcal{O}_n}$ admits a rational resolution of singularities, if $n$ is even. 
\end{cor}

\begin{proof} We have the decomposition into $\text{SO}_n$-orbits 
$$\overline{\mathcal{O}_n} = \mathcal{O}_n \coprod \mathcal{O}_\mathrm{min} \coprod \left\{0\right\}.$$
$\mathcal{O}_n$ consists of elements of rank $2$ and is of minimal rank (Lemma \ref{lmmBOrbits3}). $\mathcal{O}_\mathrm{min}$ consists of elements of rank $1$ and is of minimal rank (Lemma \ref{lemmaMinOrbit}). We conclude that any $B$-orbit closure of $\overline{\mathcal{O}_n}$ is Frobenius split (Corollary \ref{cor-FS}) and admits a resolution with connected fibers (Remark \ref{rem-proof}). Using \cite[Proposition 1.2.5]{BK}, the assertion follows. 
\end{proof}

\begin{cor}\label{corNonNormalB_r} For $n = 2r+1 \geq 5$ the $B$-orbit $Y$ corresponding to $(b_2 b_1 , f_r) \in W^P \times I_n$
has non-normal closure. 
\end{cor}

\begin{proof}
The minimal $B$-orbit $Y_0$ corresponds to $(1,f_r)$. The minimal parabolic $P_1$ raises $Y_0$ to $Y_1$ with type $U$, and the minimal parabolic $P_2$ raises $Y_1$ to $Y$ with type $U$. As $P_2$ raises $Y_0$ with type $N$, the claim follows from Lemma \ref{lmmSphericalZariski}. 
\end{proof}
 
\begin{example}\label{exm}
We describe the non-normal example from Corollary \ref{corNonNormalB_r} more explicitly. It is $n= 2r+1 \geq 5$. One has $I_r \supseteq \{ f_{r-1} , f_{r} \}$ and $W^P \supseteq \{ 1, b_1 ,b_2 b_1 \}$. The minimal parabolics corresponding to $b_1$ and $b_2$ we denote by $P_1$ and $P_2$.
$$
\xymatrix{
& (1,f_r) \ar@{-}[dl]_{P_1} \ar@{=}[dr]^{P_2} & \\
(b_1 ,f_r ) \ar@{-}[d]_{P_2}&  & (1, f_{r-1} ) \ar@{-}[d]^{P_1}  \\
(b_2 b_1 , f_r ) \ar@{=}[dr]_{P_1} &  & (b_1 , f_{r-1} ) \ar@{-}[dl]^{P_2}  \\
& (b_2 b_1 , f_{r-1} ) & \\
 }
$$
\begin{center}
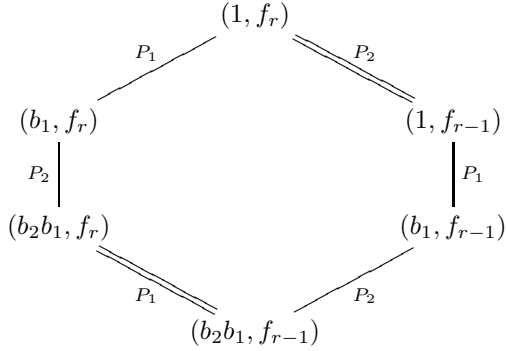
\textbf{Figure 3. Excerpt of weak order in case $(B_r ,[3,1^{2r-2}])$}
\end{center}
A simple (resp. double) edge indicates that the covering relation is of type $U$ (resp. of type $N$). By Lemma \ref{lmmSphericalZariski} the orbit closure $Z= \overline{B \cdot b_2 b_1 \cdot f_r}$ is non-normal along $Y=\overline{B \cdot f_{r-1}}$. As subsets of $\mathfrak{so}_{n}$ one may compute
$$Z = \left\{ \left(
\begin{array}{ccc}
0& -y^T J_3&0\\
0&   C      &y\\
0&   0      &0   
\end{array} \right) \mid C \in \cN (\mathfrak{so}_3 ), C^2 y = 0 \right\}, \quad Y = \left\{ (C,y) \in Z \mid C=0 \right\}.$$
Variety $Z$ is $4$-dimensional. The reduced ideal for $Z$ is generated by the polynomials
$X_1^2 +2 X_2 X_3$ and  $X_1 Y_2 -X_2 Y_3 - X_3 Y_1$. Hence $Z$ is a complete intersection which is singular along the $3$-dimensional affine space $Y= V(X_1,X_2 ,X_3) \subseteq Z$. 
\end{example}

\begin{remark} The variety $\overline{\mathcal{O}(c_1 c_2 ,1)}$ from Example \ref{exampleNon-normal/Non-CM} is isomorphic to
$Z$. 
\end{remark}

\section{The cases $G_2$ and $F_4$}

\subsection{The group $G=G_2$} Let $\mathfrak{h} \subseteq \g = \mathfrak{g}_2$ be a Cartan subalgebra and $\Delta = \{ \alpha_1 , \alpha_2 \} \subseteq \mathfrak{h}^*$ be a set of simple roots with $\alpha_1$ short.
A system of positive roots for
$\g =\mathfrak{g}_2$ is then $\Phi^+ = \{\alpha_1 , \alpha_2 , \alpha_1 +\alpha_2 , 2 \alpha_1 + \alpha_2 , 3 \alpha_1 + \alpha_2 ,\beta:= 3 \alpha_1 + 2 \alpha_2 \}$.
For a root $\rho \in \Phi$ pick a generator $X_\rho$ of the vector space $\g_\rho$. There is one nilpotent $G_2$-orbit of height $2$. It is the minimal nilpotent orbit $\mathcal{O}_{\mathrm{min}} =\mathrm{G} \cdot X_\beta$, which is $6$-dimensional.

\medskip

\noindent We can use Lemma \ref{lemmaMinOrbit}. At first, any $B$-orbit closure of $\overline{\mathcal{O}_{\mathrm{min}}}$ admits a rational resolution of singularities. Further, there are $\mid \Phi_{\mathrm{lg}} \mid = 6$ many Borel orbits in $\mathcal{O}_{\mathrm{min}}$. The minimal Borel orbit is $B \cdot X_\beta = \g_\beta - \{0\}$.
Representatives of $B$-orbits  and the (linear) closure graph are depicted in the following line.
$$\xymatrix{
X_\beta &\ar@{-}[l]_{P_{\alpha_2}} X_{3 \alpha_1 + \alpha_2} &\ar@{-}[l]_{P_{\alpha_1}} X_{\alpha_2} &\ar@{-}[l]_{P_{\alpha_2}} X_{-\alpha_2}&\ar@{-}[l]_{P_{\alpha_1}} X_{-3 \alpha_1 - \alpha_2} &\ar@{-}[l]_{P_{\alpha_2}} X_{-\beta} .\\
}$$

\begin{remark} There is one nilpotent $G_2$-orbit of height $3$ (hence spherical). It is the $8$-dimensional orbit $\mathcal{O}_8 = G_2 \cdot X_{2 \alpha_1 + \alpha_2}$. According to \cite{levsmi} the $G_2$-closure $\overline{\mathcal{O}_8}$ is not normal. Two further Borel orbit closures of $\overline{\mathcal{O}_8}$ are investigated, namely the irreducible components of the scheme $\overline{\mathcal{O}_8} \cap \mathfrak{b}$ which are 
$\overline{P_{\alpha_2} \cdot X_{\alpha_1}}$ and $\overline{P_{\alpha_1} \cdot X_{\alpha_1 + \alpha_2}}$. It is shown in \textit{loc.cit.} that the latter is not normal.
\end{remark}

\subsection{The group $G= F_4$} Let $\mathfrak{h} \subseteq \g := \mathfrak{f}_4$ be a Cartan subalgebra and $\Delta = \{ \alpha_1 , \alpha_2 , \alpha_3 , \alpha_4 \} \subseteq \mathfrak{h}^*$ be a set of simple roots with $\alpha_1 , \alpha_2$ short. 
The highest root is $\beta = 2 \alpha_1 + 4 \alpha_2 + 3 \alpha_3 + 2 \alpha_4$. There are two nilpotent $G$-orbits of height $2$. The minimal orbit
$\mathcal{O}_\mathrm{min} = G \cdot X_\beta$ and $\mathcal{O}_2 = G \cdot X_{\beta -\alpha_2 -  \alpha_3 -\alpha_4}$.
In the following, we use the notation $(abcd)$ for the root $a \alpha_1 + b \alpha_2 + c \alpha_3 + d \alpha_4$.

\subsubsection{$\mathcal{O}_\mathrm{min}$} We use Lemma \ref{lemmaMinOrbit}. Any $B$-orbit closure of $\overline{\mathcal{O}_\mathrm{min}}$ has a rational resolution of singularities.
There are $24$ long roots:
$$\pm\{\beta ,(2431),(2421),(2221),(2211),(0221),(2210),(0211),(0210),(0011),\alpha_3 , \alpha_4 \}.$$
The corresponding root vectors $X_{\rho}$ represent $B$-orbits of $\mathcal{O}_{\mathrm{min}}$.
The closure graph is the reversed Hasse diagram of $\Phi_{\mathrm{lg}}$.

\subsubsection{$\mathcal{O}_2$} The Dynkin characteristic is $\Omega(\mathcal{O}_2 )=(\alpha_1 (H), \alpha_2 (H), \alpha_3 (H), \alpha_4 (H))=(1,0,0,0)$. 
For the $\ad_H$-eigenspace decomposition one has
$$\g(0)= \langle \g_\rho \mid \rho = f_2 \alpha_2 + f_3 \alpha_3 + f_4 \alpha_4 \rangle \oplus \mathfrak{h} = \mathfrak{so}_7 \oplus \kk ,$$
$$\g(2) = \langle g_{\rho} \mid \rho \succ 2 \alpha_1, \quad \rho \in \Phi^+ \rangle \simeq \kk^7 .$$

We have an equivariant resolution $G \times^{P^{\alpha_1}} \g(2) \to \overline{\mathcal{O}_2}$. 
Let us apply the results from paragraph \ref{3Nilp}. In particular, we saw that $\g(2)$ is not of minimal rank.
Denote by $L = \mathrm{Spin}(7) \times \G_m$ the Levi factor of $P^{\alpha_1}$. There are four $B_L$-orbits in the dense $L$-orbit $L \cdot X_{(2321)}$ of $\g (2)$. They are of dimension $4,5,6$ and $7$. The minimal $B_L$-orbit is 
$$B_L \cdot X_{(2321)} = \kk^* X_{(2321)} + \sum_{\rho \succ (2321)} \g_{\rho} \simeq \kk^* \times \kk^3 .$$
The remaining $B_L$-orbits of $\g(2)$ are of rank $2$. The weak order on $B_L \setminus \g(2)$ is depicted in the left column of the following figure. 
$$
\xymatrix{
X_{(2321)}\ar@{=}[d]_{P_{\alpha_2}} \ar@{-}[dr]^{P_{\alpha_1}} &   \\
X_{(2421)} + X_{(2221)} \ar@{-}[d]_{P_{\alpha_3}}& X_{(1321)} \ar@{-}[d]_{P_{\alpha_2}} \\
X_{(2431)} + X_{(2211)}  \ar@{-}[d]_{P_{\alpha_4}} & X_{(1221)}\\
X_{\beta} + X_{(2210)} & 
}$$

\begin{center}
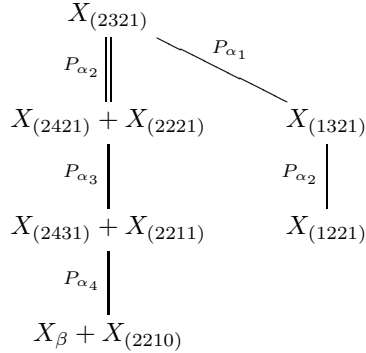
\textbf{Figure 4. Excerpt of weak order in case $(F_4, \mathcal{O}_2 )$ } 
\end{center}

A double edge (resp. single edge) indicates that the covering relation is of Type $N$ (resp. Type $U$).
From Lemma \ref{lmmSphericalZariski} we derive that the closure of the orbit $B \cdot X_{(1221)}$ is not normal
along $B \cdot (X_{(2421)} + X_{(2221)})$.


\begin{thebibliography}{99}
\bibitem{piotr} Achinger, P., Perrin, N., \textit{Spherical multiple flags}. Preprint arXiv:1307.7236. To appear in Advanced Studies in Pure Mathematics.
\bibitem{Boos-Reineke} Boos, M., Reineke, M., 
\textit{$B$-orbits of $2$-nilpotent matrices and generalizations}. 
Highlights in Lie algebraic methods, Progr. Math., 295, Birkh\"auser/Springer, New York, 2012, 147-166.
\bibitem{bourbaki} Bourbaki, N. {\it Groupes et alg{\`e}bres de Lie, Chapitres IV, V \& VI}. Hermann 1968.
\bibitem{brion} Brion, M., \textit{Quelques propri{\'e}t{\'e}s des espaces homog{\`e}nes sph{\'e}riques}. Manuscripta Math. {\bf 55} (1986), no. 2, 191--198.
\bibitem{brion1} \bysame, 
\textit{On orbit closures of spherical subgroups in flag varieties}. 
Comment. Math. Helv. {\bf 76} (2001), no. 2, 263--299.  
\bibitem{brion2} \bysame, 
\textit{Multiplicity-free subvarieties of flag varieties}. 
Contemp. Math., {\bf 331} (2003), 13--23. 
\bibitem{BK} Brion, M., Kumar, S., 
\textit{Frobenius splitting methods in geometry and representation theory}. 
Progress in Mathematics, 231. Birkh\"auser Boston, Inc., Boston, MA, 2005. 
\bibitem{cmp1} P.-E. Chaput, L.~Manivel, and N.~Perrin, 
\emph{Quantum cohomology of minuscule homogeneous spaces},
Transform. Groups \textbf{13} (2008), no.~1, 47--89. 
\bibitem{cmp2} \bysame, 
\emph{Quantum cohomology of minuscule homogeneous spaces. II. Hidden
  symmetries.} IMRN 2007, no. \textbf{22}. 
\bibitem{cmp3} \bysame, 
\emph{Quantum cohomology of minuscule homogeneous spaces III :
  semi-simplicity and consequences}, Canad. J. Math. \textbf{62}
(2010), no. 6, 1246--1263. 
\bibitem{CP} P.-E. Chaput and N.~Perrin, 
\emph{On the quantum cohomology of adjoint varieties}. 
Proc. Lond. Math. Soc. (3) \textbf{103} (2011), no. 2, 294--330.
\bibitem{DFZJ} Di Francesco, P., Zinn-Justin, P.,
\textit{From orbital varieties to alternating sign matrices}.
arXiv:0512047.
\bibitem{donkin} Donkin, S., 
\textit{The normality of closures of conjugacy classes of matrices}. 
Invent. math. {\bf 101} (1990), no. 3, 717--736.
\bibitem{FR} Fowler, R., R\"ohrle, G.,
\textit{Spherical nilpotent orbits in positive characteristic}. Pacific J. Math. {\bf 237} (2008), no. 2, 241--286.
\bibitem{fulton} Fulton, W.,
\textit{Schubert varieties, degeneracy loci and determinantal varieties}.
Duke Math. J. {\bf 65} (1992), no. 3, 381--420.
\bibitem{HT} He, X., Thomsen, J.F., 
\textit{Geometry of $B\times B$-orbit closures in equivariant embeddings}.  
Adv. Math. {\bf 216} (2007),  no. 2, 626--646.
\bibitem{hesselink} Hesselink, W., 
\textit{The normality of closures of orbits in a Lie algebra}.
Commentarii Mathematici Helvetici {\bf 54 } (1979), no. 1, 105--110.
\bibitem{juteau} Juteau, D.,
\textit{Cohomology of the minimal nilpotent orbit}.
Transform. Groups. {\bf 13} (2014), no.2, 355--387.
\bibitem{kac} Kac, V., 
\textit{Some remarks on nilpotent orbits}. 
J. Alg. 64 (1980), 190--213.
\bibitem{KLS} Knutson, A., Lam, T., Speyer, D., \textit{Projections of Richardson varieties}. J. Reine Angew. Math. {\bf 687} (2014), 133--157.
\bibitem{levsmi} Levasseur, T., Smith, S. P.,
\textit{Primitive ideals and nilpotent orbits in Type $G_2$}.
J. Alg. 114 (1988), 81--105.
\bibitem{Melnikov} Melnikov, A., 
\textit{Description of $B$-orbit closures of order $2$ in upper triangular matrices}. 
Transf. Groups {\bf 11} (2006), no.2, 217--247.
\bibitem{mvdk} Mehta, V., van der Kallen, W.,  
\textit{A simultaneous Frobenius splitting for closures of conjugacy classes of nilpotent matrices}.
Compositio Math. {\bf 84} (1992), no. 2, 211--221.
\bibitem{Panyushev} Panyushev, D. I., 
\textit{Complexity and nilpotent orbits}. 
Manuscripta Mathematica, {\bf 83 }, (1994), no. 3-4, 223--237. 
\bibitem{survey} Perrin, N., 
\textit{Geometry of spherical varieties}. 
Transform. Groups. {\bf 19} (2014), no.1, 171--223.
\bibitem{embeddings} Perrin, N., 
\textit{Compatibly split subvarieties of group embeddings}. 
 Trans. Amer. Math. Soc. \textbf{367} (2015), no. 12, 8421--8438.
\bibitem{PinPhD} Pin, S.,
\textit{Adh\'{e}rences d'orbites de sous-groupes de Borel dans les espaces symm\'{e}trique}.
PhD Thesis Universit\'{e} Joseph Fourier Grenoble, 2001.
\bibitem{ressayre} Ressayre, N., 
\textit{Spherical homogeneous spaces of minimal rank}. 
Adv. Math. {\bf 224} (2010), no. 5, 1784--1800. 
\bibitem{RS1} Richardson, R.W., Springer, T.A., 
\textit{The Bruhat order on symmetric varieties}. 
Geom. Dedicata {\bf 35} (1990), no. 1-3, 389--436.
\bibitem{Rothbach} Rothbach, B.D., 
\textit{Borel Orbits of $X^2 = 0$ in $\mathfrak{gl}_n$}. 
PhD Thesis, Berkeley 2009.
\bibitem{SpringerSteinberg} Springer, T. A., Steinberg, R. \textit{Conjugacy classes}. 1970 Seminar on Algebraic Groups and Related Finite Groups (The Institute for Advanced Study, Princeton, N.J., 1968/69) 167--266. 
\bibitem{vinberg} Vinberg, {\`E}. \textit{Complexity of actions of reductive groups}. Funktsional. Anal. i Prilozhen. {\bf 20} (1986), no. 1, 1--13, 96.
\bibitem{Wyser} Wyser, B. J.,
\textit{Symmetric subgroup orbit closures on flag varieties: Their equivariant geometry, combinatorics, and connections with degeneracy loci}.
PhD Thesis, University of Georgia, 2012.
\bibitem{xiaoshu} Xiao, H., Shu, B., 
\textit{Normality of orthogonal and symplectic nilpotent orbit closures in positive characteristic}.
Journal of Algebra, {\bf 443}, (2015), 33-48.
\end{thebibliography}
\end{document}